\documentclass[12pt]{amsart}
\usepackage{amsmath,amssymb,latexsym,cancel,rotating}
\usepackage{graphicx,amssymb,mathrsfs,amsmath,color,fancyhdr,amsthm,verbatim}
\usepackage[all]{xy}
\usepackage{tikz}

\newtheorem{theorem}{Theorem}[section]
\newtheorem{lemma}[theorem]{Lemma}
\newtheorem{corollary}[theorem]{Corollary}

\newtheorem{proposition}[theorem]{Proposition}

 \def\cD{{\mathcal D}}

  \def\bbZ{{\mathbb Z}}  \def\bbQ{{\mathbb Q}}
    \def\bbF{{\mathbb F}}
    
  \def\bbV{{\mathbb V}}
\def\bbW{{\mathbb W}}

         \def\bfU{{\bf U}}
   
  \def\leq{\leqslant}  \def\geq{\geqslant}

\def\Hom{\mbox{\rm Hom}}

\def\dim{\mbox{\rm dim}\,}

\def\bfV{{\mathbf V}}
\def\bfW{{\mathbf W}}

\begin{document}

\title[Geometric realizations of Ringel-Hall algebras]
{Geometric realizations of Ringel-Hall algebras of continuous quivers of type $A$}

\author[Zhao]{Minghui Zhao}
\address{School of Science, Beijing Forestry University, Beijing 100083, P. R. China}
\email{zhaomh@bjfu.edu.cn (M.Zhao)}


\date{\today}


\bibliographystyle{abbrv}

\begin{abstract}
Lusztig introduced the geometric realizations of quantum groups associated to finite quivers and defined their canonical bases. Sala and Schiffmann introduced the Ringel-Hall algebra of line and realized it as the direct limit of Ringel-Hall algebras of finite quivers of type $A$. In this paper, we shall give geometric realizations of Ringel-Hall algebras of continuous quivers of type $A$ via the geometric realizations of Lusztig by using the method of approximation given by Sala and Schiffmann.
\end{abstract}

\maketitle


\section{Introduction}

\subsection{}

The Ringel-Hall algebra of the category of representations of a finite quiver $Q=(Q_0,Q_1)$ was introduced by Ringel in \cite{Ringel_Hall_algebras_and_quantum_groups}.  The composition subalgebra of the Ringel-Hall algebra is isomorphic to the negative part $\bfU^-$ of the quantum group associated to the quiver $Q$ (\cite{Ringel_Hall_algebras_and_quantum_groups,Green1995,XIAO1997100}). 

Inspired by the work of Ringel,
Lusztig gave a geometric realization of $\bfU^-$ in 
\cite{Lusztig_Canonical_bases_arising_from_quantized_enveloping_algebra,Lusztig_Quivers_perverse_sheaves_and_the_quantized_enveloping_algebras}.
Fix a $Q_0$-graded vector space $\bbV=\bigoplus_{x\in Q_0}\bbV_x$ with dimension vector $\nu$ and denote by $E_{\bbV}$ the variety consisting of representations of the quiver $Q$.
There is a $G_{\bbV}=\prod_{x\in Q_0}GL(\bbV_x)$ action on $E_{\bbV}$.
Denote by $\cD_{G_\bbV}^b(E_\bbV)$ the $G_\bbV$-equivariant bounded derived category of $\overline{\bbQ}_l$-constructible mixed complexes on $E_\bbV$.
Let $K(\mathcal{Q}_{\bbV})$ be the Grothendieck group of an additive subcategory $\mathcal{Q}_{\bbV}$ of $\cD_{G_\bbV}^b(E_\bbV)$.
Lusztig proved that $\bigoplus_{\nu}K(\mathcal{Q}_{\bbV})$ is isomorphic to the integral form of $\bfU^-$. By this isomorphism, the isomorphism classes of simple perverse sheaves in $\mathcal{Q}_{\bbV}$ provide a basis of the $\bfU^-$ (\cite{Lusztig_Canonical_bases_arising_from_quantized_enveloping_algebra,Lusztig_Quivers_perverse_sheaves_and_the_quantized_enveloping_algebras,Lusztig_Introduction_to_quantum_groups}).
Xiao, Xu and Zhao generalized the geometric construction to the generic form of the whole Ringel-Hall algebra in \cite{XXZ2022254}.

Lusztig also gave constructions of Ringel-Hall algebras via functions in \cite{Lusztig_Canonical_bases_and_Hall_algebras}. Denoted by
$\mathcal{F}_{G^F_\bbV}(E^F_\bbV)$
the set of $G^F_\bbV$-invariant finitely supported $\mathbb{C}$-valued functions on $E^F_\bbV$, where $G^F_\bbV$ and $E^F_\bbV$ are closed points in $G_\bbV$ and $E_\bbV$ fixed by Frobenius morphisms.
Let $\mathcal{F}_{Q}=\bigoplus_{\nu}\mathcal{F}_{G^F_\bbV}(E^F_\bbV)$.
Lusztig showed that $\mathcal{F}_{Q}$ is isomorphic to the Ringel-Hall algebra of $Q$.
According to the sheaf-function correspondence (\cite{Kiehl_Weissauer_Weil_conjectures_perverse_sheaves_and_l'adic_Fourier_transform}), there is a homomorphism of algebras $\chi^{F}:K_{Q}\rightarrow\mathcal{F}_{Q}$.

\subsection{}
Representation theory of continuous quivers of type $A$ is important in persistent homology and widely used in topological data analysis (\cite{Persistence}).
In \cite{Crawley2015Decomposition,2018Decomposition}, Crawley-Boevey and Botnan gave a
classification of indecomposable representations of $\mathbb{R}$.
In \cite{2017Interval}, Botnan gave a
classification of indecomposable representations of infinite zigzag.
Igusa, Rock and Todorov introduced general continuous quivers of type
$A$ and classified indecomposable representations (\cite{Igusa2022Continuous,rock2023A_2,IGUSA_ROCK_TODOROV_2022,rock2023A_3}).
Hanson and Rock introduced continuous quivers of type $\tilde{A}$ and gave a
classification of indecomposable representations (\cite{2020Decomposition_Hanson_Rock}). 
Rock and Zhu studied Nakayama representations of this quiver
in \cite{rock2023continuous_Nakayama}.
In \cite{Appel_2020,Appel_2022,Sala_2019,Sala_2021}, Appel, Sala
and Schiffmann introduced continuum quivers, continuum Kac–Moody algebras and continuum quantum groups associated to continuum quivers.

In \cite{Sala_2019,Sala_2021}, Sala and Schiffmann showed that the continuum quantum group associated to the circle can be realized as the direct limit of a direct system of quantum groups of type $\tilde{A}$.
This result also holds for quantum group of the line (a continuous quiver of type $A$).
In \cite{Sala_2019,Sala_2021}, Sala and Schiffmann introduced the Ringel-Hall
algebras associated to the circle and the line. They also studied  relations between Ringel-Hall
algebras and continuum quantum groups.

\subsection{}

In this paper, we shall give a geometric realization of the Ringel-Hall algebra of a continuous quiver of type $A$ by using the method of approximation in \cite{Sala_2019,Sala_2021} and the geometric realizations of quantum groups of type $A$ in \cite{Lusztig_Canonical_bases_arising_from_quantized_enveloping_algebra,Lusztig_Quivers_perverse_sheaves_and_the_quantized_enveloping_algebras,Lusztig_Introduction_to_quantum_groups}.

Let $A_\mathbb{R}$ be a continuous quiver of type $A$ and 
$\mathcal{I}$ a partition of $\mathbb{R}$.
We define a finite quiver $Q_{\mathcal{I}}$ of type $A$ associated to $A_{\mathbb{R}}$ and $\mathcal{I}$.
Similarly to the method of approximation in \cite{Sala_2019,Sala_2021}, we introduce a $\mathbb{Z}[v,v^{-1}]$-algebra $K_{A_\mathbb{R}}$ as a direct limit of a direct system  $(K_{Q_{\mathcal{I}}},\Phi_{\mathcal{I}_1,\mathcal{I}_2})$. 
Let $\mathcal{F}_{A_\mathbb{R}}$ be the Ringel-Hall algebra associated to $A_\mathbb{R}$.
According to the sheaf-function correspondence, there is a homomorphism of $\mathbb{Z}$-algebras $\chi_{A_\mathbb{R}}: K_{A_\mathbb{R}}\rightarrow\mathcal{F}_{A_\mathbb{R}}$. The algebra $K_{A_\mathbb{R}}$ can be viewed as a geometric version of $\mathcal{F}_{A_\mathbb{R}}$.
In this paper, we also define a canonical basis of a completion  $\bar{K}_{A_{\mathbb{R}}}$ of $K_{A_\mathbb{R}}$.

\subsection{}

In Section 2, we shall recall the definitions of Ringel-Hall algerbas of finite quivers and their geometric realizations.
In Section 3, the concepts of continuous quivers of type $A$ and their representations are recalled. In this section, we also
define a finite quiver $Q_{\mathcal{I}}$ of type $A$ associated to a continuous quivers $A_{\mathbb{R}}$ and a partition $\mathcal{I}$.
In Section 4, 
we recall the Ringel-Hall algebras $\mathcal{F}_{A_\mathbb{R}}$ associated to a continuous quiver $A_\mathbb{R}$ of type $A$. Then the geometric version $K_{A_\mathbb{R}}$ of $\mathcal{F}_{A_\mathbb{R}}$ is defined.
The completion $\bar{K}_{A_\mathbb{R}}$ of 
$K_{A_\mathbb{R}}$ and its canonical basis are defined in Section 5.

\section{Ringel-Hall algebras and canonical bases}

In this section, we shall recall the definitions of Ringel-Hall algerbas of finite quivers (\cite{Ringel_Hall_algebras_and_quantum_groups,Schiffmann_Lectures1}) and their geometric realizations (\cite{Lusztig_Canonical_bases_arising_from_quantized_enveloping_algebra,Lusztig_Quivers_perverse_sheaves_and_the_quantized_enveloping_algebras,Lusztig_Introduction_to_quantum_groups,Schiffmann_Lectures2}).

\subsection{}
A quiver $Q=(Q_0,Q_1,s,t)$ is consisting of a set $Q_0$ of vertices, a set $Q_1$  of arrows and two maps $s,t:Q_1\rightarrow Q_0$ such that $s(h)$ is the source of $h$ and $t(h)$ is the target of $h$ for any $h\in Q_1$.

Let $K$ be a fixed field. A representation $V=(V(x),V(h))$ of $Q$ over ${K}$ is consisting of
a ${K}$-vector space $V(x)$ for any $x\in Q_0$ and
a ${K}$-linear map $V(h):V(s(h))\rightarrow V(t(h))$ for any $h\in Q_1$.

The collection of linear maps $f_{x}:V(x)\rightarrow W(x)$ for all $x\in Q_0$ is called a map from $V$ to $W$, if $f_{t(h)}V(h)=W(h)f_{s(h)}$ for all $h\in Q_1$.

Denote by $\textrm{Rep}_{{K}}(Q)$ the category of ${K}$-linear representations of $Q$ and
by $\textrm{rep}_{{K}}(Q)$ the subcategory of finitely dimensional representations.

A function $\nu:Q_0\rightarrow\mathbb{Z}_{\geq0}$ is called the dimension vector of a representation $V=(V(x), V(h))$, if $\dim V(x)=\nu(x)$ for all $x\in Q_0$. In this case, we denote $\underline{\dim}V=\nu$.

Fix a dimension vector $\nu$ and a $Q_0$-graded ${K}$-vector space $\bbV=\oplus_{x\in Q_0}\bbV_x$ such that $\dim\bbV_x=\nu(x)$.
In this case, we also denote $\underline{\dim}\bbV=\nu$.
Let
$$E_\bbV=E_{Q,\bbV}=\prod_{h\in Q_1}\Hom_{{K}}(\bbV_{s(h)},\bbV_{t(h)}).$$
The group
$G_\bbV=\prod_{x\in Q_0}GL_{K}(\bbV_x)$ acts on $E_\bbV$ by $g.\phi=g\phi g^{-1}$ for any $g\in G_\bbV$ and $\phi\in E_\bbV$.

\begin{proposition}[\cite{Crawley-Boevey90}]\label{one-to-one}
There is an one to one correspondence between the $G_\bbV$-orbits in $E_{\bbV}$
and the isomorphism classes of finitely dimensional representations of $Q$. 
\end{proposition}

For any representation $V$, let $\mathcal{O}_{V}$ be the orbit corresponding to $[V]$.

\subsection{}

In this section, we follow the notations in \cite{Kiehl_Weissauer_Weil_conjectures_perverse_sheaves_and_l'adic_Fourier_transform,Bernstein_Lunts_Equivariant_sheaves_and_functors}.

Let $\bbF_q$ be a finite field with $q$ elements and ${K}$ the algebraic closure of $\bbF_q$.
Let $G$ be an algebraic group over $K$ and $X$ a scheme of finite type over $K$ together with a $G$-action.
Assume that $X$ and $G$ have $\mathbb{F}_q$-structures and let $F_{X}: X\rightarrow X$ and $F_{G}: G\rightarrow G$ be the Frobenius morphisms. 
Let $G^F$ and $X^F$ be the sets of closed points of $G$ and $X$ fixed by  the Frobenius morphisms.

Denote by $\mathcal{D}_G^b(X)$ the $G$-equivariant bounded derived category of $\overline{\mathbb{Q}}_l$-constructible mixed complexes on $X$.
Let $K_{G}(X)$ be the Grothendieck group of $\mathcal{D}_{G}^b(X)$.
Denoted by
$\mathcal{F}_{G^F}(X^F)$
the set of $G^F$-invariant finitely supported $\mathbb{C}$-valued functions on $X^F$.
According to the sheaf-function correspondence (\cite{Kiehl_Weissauer_Weil_conjectures_perverse_sheaves_and_l'adic_Fourier_transform}), there is a map $$\chi^{F}: K_{G}(X)\rightarrow\mathcal{F}_{G^F}(X^F)$$ sending $\mathcal{L}$ to $\chi_{\mathcal{L}}^{F}$.

\subsection{}
Let $Q=(Q_0,Q_1)$ be a quiver. Fix a dimension vector $\nu$ and a $Q_0$-graded ${K}$-vector space $\bbV$ such that $\underline{\dim}\bbV=\nu$.
Note that $\bbV$, $E_\bbV$ and $G_\bbV$ have $\bbF_q$-structures.
Let $F_\bbV:\bbV\rightarrow\bbV$, $F_{E_\bbV}:E_\bbV\rightarrow E_\bbV$ and $F_{G_\bbV}:G_\bbV\rightarrow G_\bbV$ be the corresponding Frobenius morphisms. Let $\bbV^F$, $E_\bbV^F$ and $G_\bbV^F$ be the sets of closed points of $\bbV$, $E_\bbV$ and $G_\bbV$ fixed by the Frobenius morphisms, respectively.

For any $\nu,\nu',\nu''$ such that $\nu=\nu'+\nu''$, fix $Q_0$-graded ${K}$-vector spaces $\bbV$, $\bbV'$, $\bbV''$ with dimension vectors $\nu,\nu',\nu''$, respectively.

Consider the following diagram 
\begin{equation}\label{multi-L}
\xymatrix{E_{\bbV'}\times{E}_{\bbV''}&{E}'\ar[l]_-{p_1}\ar[r]^-{p_2}&{E}''\ar[r]^-{p_3}&{E}_{\bbV}},    
\end{equation}
where
\begin{enumerate}
  \item[(1)]${E}''=\{(\phi,\bbW)\}$, where $\phi\in {E}_{\bbV}$ and $\bbW$ is a $\phi$-stable (i.e. $\phi_h(\bbW_{s(h)})\subset\bbW_{t(h)}$ for any $h\in Q_1$) subspace of $\bbV$ with dimension vector $\nu''$;
  \item[(2)]${E}'=\{(\phi,\bbW,R'',R')\}$, where $(\phi,\bbW)\in {E}''$, $R'':\bbV''\simeq \bbW$ and $R':\bbV'\simeq{\bbV}/\bbW$;
  \item[(3)]$p_1(\phi,\bbW,R'',R')=(\phi',\phi'')$, where $\phi'$ and $\phi''$ are induced by the following commutative diagrams
      $$\xymatrix{\bbV'_{s(h)}\ar[r]^-{\phi'_{h}}\ar[d]^-{R'_{s(h)}}&\bbV'_{t(h)}\ar[d]^-{R'_{t(h)}}\\
      (\bbV/\bbW)_{s(h)}\ar[r]^-{\phi_{h}}&(\bbV/\bbW)_{t(h)}}$$ and
      $$\xymatrix{\bbV''_{s(h)}\ar[r]^-{\phi''_{h}}\ar[d]^-{R''_{s(h)}}&\bbV''_{t(h)}\ar[d]^-{R''_{t(h)}}\\
      \bbW_{s(h)}\ar[r]^-{\phi_{h}}&\bbW_{t(h)}}$$
      for all $h\in Q_1$;
  \item[(4)]$p_2(\phi,\bbW,R'',R')=(\phi,\bbW)$;
  \item[(5)]$p_3(\phi,\bbW)=\phi$.
\end{enumerate}

The group $G_{\bbV}$ acts on $E''$ by $g.(\phi, \bbW)=(g\phi g^{-1}, g\bbW)$ for any $g\in G_{\bbV}$.
The groups $G_{\bbV'}\times G_{\bbV''}$ and $G_{\bbV}$ act on $E'$ by $(g_1, g_2).(\phi,\bbW,R'',R')=(\phi, \bbW, g_2R'', g_1R')$ and $g.(\phi,\bbW, R'',R')=(g\phi g^{-1}, g\bbW, R''g^{-1}, R'g^{-1})$ for any $(g_1, g_2)\in G_{\bbV'}\times G_{\bbV''}$ and $g\in G_{\bbV}$. The map $p_1$ is $G_{\bbV'}\times G_{\bbV''}\times G_{\bbV}$-equivariant ($G_{\bbV}$ acts on $E_{\bbV'}\times E_{\bbV''}$ trivially)  and $p_2$ is a principal $G_{\bbV'}\times G_{\bbV''}$-bundle.

Since $p_2$ is a principal $G_{\bbV'}\times G_{\bbV''}$-bundle, there is an equivalence $$p^\ast_2:\mathcal{D}^b_{G_\bbV}(E'')\rightarrow\mathcal{D}^b_{G_\bbV\times G_{\bbV'}\times G_{\bbV''}}(E').$$ The inverse of $p^\ast_2$ is denoted by $${p_2}_\flat:\mathcal{D}^b_{G_\bbV\times G_{\bbV'}\times G_{\bbV''}}(E')\rightarrow\mathcal{D}^b_{G_\bbV}(E'').$$
Then we can define the induction functor
\begin{eqnarray*}
\ast:\mathcal{D}^b_{G_{\bbV'}}(E_{\bbV'})\times\mathcal{D}^b_{G_{\bbV''}}(E_{\bbV''})&\rightarrow&\mathcal{D}^b_{G_\bbV}(E_\bbV)\\
(\mathcal{L}',\mathcal{L}'')&\mapsto&\mathcal{L}'\ast\mathcal{L}''={p_3}_{!}{p_2}_\flat p_1^{\ast}(\mathcal{L}'\boxtimes\mathcal{L}'').
\end{eqnarray*}

Let 
$$K_{{Q}}=\bigoplus_{\nu\in\mathbb{N}{Q_0}}K_{G_\bbV}(E_\bbV),$$
which is a $\mathbb{Z}[v,v^{-1}]$-module by $v^{\pm1}[\mathcal{L}]=[\mathcal{L}[\pm1](\pm\frac{1}{2})]$, where $[n]$ is the shift functor and $(\frac{n}{2})$ is the Tate twist.

\begin{proposition}[\cite{Lusztig_Canonical_bases_arising_from_quantized_enveloping_algebra,Lusztig_Quivers_perverse_sheaves_and_the_quantized_enveloping_algebras,Lusztig_Introduction_to_quantum_groups}]
The induction functors for various $\nu$, $\nu'$, $\nu''$ induce an associative multiplication structure $$\ast:K_{{Q}}\times K_{{Q}}\rightarrow K_{{Q}}$$ sending $(\mathcal{L}',\mathcal{L}'')$ to $\mathcal{L}'\ast\mathcal{L}''$.
\end{proposition}

\subsection{}
Similarly to (\ref{multi-L}), consider the following diagram over $\mathbb{F}_q$
$$
\xymatrix{E^F_{\bbV'}\times{E}^F_{\bbV''}&{E}'^F\ar[l]_-{p_1}\ar[r]^-{p_2}&{E}''^F\ar[r]^-{p_3}&{E}^F_{\bbV}}.
$$

Since $p_2$ is a principal $G^F_{\bbV'}\times G^F_{\bbV''}$-bundle, there is bijection $$p^\ast_2:\mathcal{F}_{G^F_\bbV}(E''^F)\rightarrow\mathcal{F}_{G^F_\bbV\times G^F_{\bbV'}\times G^F_{\bbV''}}(E'^F).$$
At the same time, $|p_3^{-1}(\phi)|<\infty$.
Then we can define the induction map
\begin{eqnarray*}
\ast:\mathcal{F}_{G^F_{\bbV'}}(E^F_{\bbV'})\times\mathcal{F}_{G^F_{\bbV''}}(E^F_{\bbV''})&\rightarrow&\mathcal{F}_{G^F_\bbV}(E^F_\bbV)\\
(f',f'')&\mapsto&f'\ast f''={p_3}_{!}({p_2}^{\ast})^{-1}\ p_1^{\ast}(f'\otimes f'').
\end{eqnarray*}

Let $$\mathcal{F}_{Q}=\bigoplus_{\nu\in\mathbb{N}{Q_0}}\mathcal{F}_{G^F_\bbV}(E^F_\bbV).$$

\begin{proposition}[\cite{Lusztig_Canonical_bases_and_Hall_algebras}]
The induction maps for various $\nu$, $\nu'$, $\nu''$ induce an associative multiplication structure $$\ast:\mathcal{F}_{Q}\times \mathcal{F}_{Q}\rightarrow \mathcal{F}_{Q}$$ sending $(f',f'')$ to $f'\ast f''$.
\end{proposition}

\begin{proposition}[\cite{Schiffmann_Lectures2,XXZ2022254}]
The  map $\chi^F_Q: K_Q\rightarrow \mathcal{F}_Q$ is a homomorphism of $\mathbb{Z}$-algebras.
\end{proposition}

The embedding $\mathbb{Z}[v,v^{-1}]\rightarrow\mathbb{C}$ sending $v$ to $\sqrt{q}$ induces a $\mathbb{Z}[v,v^{-1}]$-module structure on $\mathbb{C}$.
The $\mathbb{Z}$-linear homomorphism
$\chi^F_Q$ induces an epimorphism of $\mathbb{C}$-algebras $$\chi^F_Q:K_Q\otimes_{\mathbb{Z}[v,v^{-1}]}\mathbb{C}\rightarrow\mathcal{F}_Q.$$ 

\subsection{}

Assume that $Q$ is a quiver of finite type in this section.
Let $B_\bbV$ be the set of isomorphism classes $[\mathcal{L}[n](\frac{n}{2})]$, where $\mathcal{L}$ is a simple perverse sheaf in $\mathcal{D}^b_{G_\bbV}(E_\bbV)$ and $n\in\bbZ$.

For any finitely dimensional representation $V$ of $Q$, let $\mathcal{O}_{V}$ be the orbit in $E_{\bbV}$ corresponding to $V$ and $j_V:\mathcal{O}_{V}\rightarrow E_{\bbV}$ the natural embedding. Let $$IC_{\mathcal{O}_{V}}=(j_V)_{!\ast}1_{\mathcal{O}_{V}},$$
where $1_{\mathcal{O}_{V}}$ is the constant sheaf on $\mathcal{O}_{V}$.

\begin{theorem}[\cite{Lusztig_Canonical_bases_arising_from_quantized_enveloping_algebra,Lusztig_Quivers_perverse_sheaves_and_the_quantized_enveloping_algebras,Lusztig_Introduction_to_quantum_groups}]\label{lusztig_canonical}
For any $Q_0$-graded ${K}$-vector space $\bbV$ such that $\underline{\dim}\bbV=\nu$, it holds that
$$B_\bbV=\{[IC_{\mathcal{O}_{V}}[n](\frac{n}{2})]|\underline{\dim}V=\nu, n\in\bbZ\}$$ and the set $B_\bbV$ is a $\mathbb{Z}$-basis of $K_{G_\bbV}(E_\bbV)$. 
\end{theorem}

Let $$B_Q=\bigcup_{\nu\in\mathbb{N}Q_0}B_\bbV,$$ which is a $\mathbb{Z}$-basis of $K_Q$ and called the canonical basis of $K_Q$.

\section{Continuous quivers of type $A$ and their representations}

In this section, we shall recall the definitions of continuous quivers of type $A$. We follow the notations in \cite{Igusa2022Continuous}.

\subsection{}
A continuous quiver $A_\mathbb{R}=(\mathbb{R},\mathbf{S},\prec)$ of type $A$ is consisting of
\begin{enumerate}
\item[(1)] the set $\mathbb{R}$ of real numbers;
\item[(2)] a discrete subset $\mathbf{S}=\{S_i|i=I\}$  of $\mathbb{R}$, where the index set $I$ is one of the following four sets $$\{i\in\mathbb{Z}|m\leq i\leq n\},\{i\in\mathbb{Z}|i\leq n\},\{i\in\mathbb{Z}|i\geq m\} \textrm{ and } \mathbb{Z};$$
\item[(3)] a partial order $\prec$ such that $\prec = <$ or $>$ in every intervals $$[S_i,S_{i+1}],\,(-\infty,S_{m}] \textrm{ and } [S_{n},+\infty).$$
\end{enumerate}


A representation $V= (V(x), V(x,y))$ of $A_\mathbb{R}$ over ${K}$ is consisting of
\begin{enumerate}
\item[(1)] a ${K}$-vector space $V(x)$ for any $x\in\mathbb{R}$;
\item[(2)] a ${K}$-linear map $V(x,y):V(x)\rightarrow V(y)$ for any $x \preccurlyeq y$ such that $$V(y,z)V(x,y)=V(x,z)$$ for any $x\preccurlyeq y\preccurlyeq z$.
\end{enumerate}

The collection of linear maps $f_{x}:V(x)\rightarrow W(x)$ for all $x\in\mathbb{R}$ is called a map from $V$ to $W$, if the following diagram is commutative
$$\xymatrix{
V(x)\ar[d]^-{f_x}\ar[r]^-{V(x,y)}&V(y) \ar[d]^-{f_y}\\
W(x)\ar[r]^-{W(x,y)}&{W(y)}.}$$

Denote by $\textrm{Rep}_{{K}}(A_\mathbb{R})$ the category of ${K}$-linear representations of $A_\mathbb{R}$, and by $\textrm{Rep}^{pwf}_{{K}}(A_\mathbb{R})$ the subcategory of pointwise finitely dimensional representations.
Denote by $\textrm{rep}_{{K}}(A_\mathbb{R})$ the subcategory of finitely generated representations of $A_\mathbb{R}$ defined in \cite{Igusa2022Continuous}. The category $\textrm{rep}_{{K}}(A_\mathbb{R})$ is the category of tame representations of $A_\mathbb{R}$ in \cite{Sala_2021}.

Let $\nu:\mathbb{R}\rightarrow\mathbb{Z}_{\geq0}$ be a function and $V= (V(x), V(x,y))$ be a representation of $A_\mathbb{R}$. If $\dim V(x)=\nu(x)$, then $\nu$ is called the dimension function of $V$ and we denote $\underline{\dim}V=\nu$.

\subsection{}
Fix a dimension function $\nu$ and a $\mathbb{R}$-graded ${K}$-vector space $\bfV=\oplus_{x\in\mathbb{R}}\bfV_x$ such that $\dim\bfV_x=\nu(x)$.
In this case, we also denote $\underline{\dim}\bfV=\nu$.
Let
$$\tilde{E}_\bfV=\tilde{E}_{A_{\mathbb{R}},\bfV}=\prod_{x\preccurlyeq y}\Hom_K(\bfV_{x},\bfV_{y})$$
and
$$E'_\bfV=E'_{A_{\mathbb{R}},\bfV}=
\{\phi=(\phi_{xy})_{x\preccurlyeq y}\in \tilde{E}_\bfV|\phi_{xz}=\phi_{yz}\phi_{xy}\}.$$
Let 
$E_\bfV=E_{A_{\mathbb{R}},\bfV}$ be the subset of $E'_\bfV$ consisting of $\phi$ such that $V=(\bfV,\phi)$ is a finitely generated representation in $\textrm{rep}_{{K}}(A_\mathbb{R})$.

The group
$G_\bfV=\prod_{x\in\mathbb{R}}GL_{K}(\bfV_x)$ acts on $\tilde{E}_\bfV$ by $g.\phi=g\phi g^{-1}$
for any $g\in G_\bfV$ and $\phi\in\tilde{E}_\bfV$.
Note that $G_\bfV(E'_\bfV)\subset E'_\bfV$ and $G_\bfV(E_\bfV)\subset E_\bfV$. Hence there are $G_\bfV$-actions on $E'_\bfV$ and $E_\bfV$.

Similarly to Proposition \ref{one-to-one}, we have the following proposition.

\begin{proposition}
There is an one to one correspondence between the $G_\bfV$-orbits in $E_\bfV$ and the isomorphism classes of finitely generated representations of $A_\mathbb{R}$. 
\end{proposition}

For any finitely generated representation $V$ of $A_{\mathbb{R}}$, let $\mathcal{O}_{V}$ be the orbit in $E_{\bfV}$ corresponding to $V$.

\subsection{}\label{top}
Let $A_\mathbb{R}=(\mathbb{R},\mathbf{S},\prec)$ be a continuous quiver of type $A$.

A sequence $\mathcal{I}=(I_1,I_2,\ldots,I_n)$ is called a partition of $\mathbb{R}$, if
\begin{enumerate}
\item[(1)] $I_1=(-\infty,a_{1}|$, $I_i=|a_{i-1},a_{i}|$ ($1<i<n$),
$I_n=|a_{n-1},+\infty)$ such that $a_1\leq a_2\leq \cdots\leq a_{n-1}$, where the notation $|$ is $[$, $($, $]$ or $)$;
\item[(2)]  $\mathbb{R}=\cup_{i}I_i$ and $I_i\cap I_{i+1}=\emptyset$  ($1\leq i<n$).
\end{enumerate}

For any partition $\mathcal{I}$ of $\mathbb{R}$, let 
$E_{\bfV,\mathcal{I}}$ be the subset of $E_\bfV$ consisting of $\phi$ such that $\phi_{xy}$ is an isomorphsim for any $x\preccurlyeq y$ in some $I\in\mathcal{I}$.
In \cite{Igusa2022Continuous}, it is proved that
$$E_{\bfV}=\bigcup_{\mathcal{I}}E_{\bfV,\mathcal{I}}.$$
Denote by $\tau_\mathcal{I}$ the embedding of $E_{\bfV,\mathcal{I}}$ in $E_{\bfV}$.

Fix $\phi\in E_{\bfV,\mathcal{I}}$, let $\phi_{yx}=\phi_{xy}^{-1}$ for $x\prec y$ in some $I\in\mathcal{I}$.
For any $x<y$ in some $I\in\mathcal{I}$, there exist $x_0=x<x_1<x_2<\ldots<x_s=y$ such that $x_i\prec x_{i+1}$ or $x_{i+1}\prec x_i$. Let $$\phi_{xy}=\phi_{x_{s-1}x_s}\cdots\phi_{x_1x_2}\phi_{x_0x_1}$$ and $\phi_{yx}=\phi^{-1}_{xy}$.

Let  $\mathcal{I}_1$ and $\mathcal{I}_2$ be two partitions of $\mathbb{R}$. The partition $\mathcal{I}_2$ is called a refinement of $\mathcal{I}_1$ if, for any $I\in\mathcal{I}_2$, there exists an interval $J\in\mathcal{I}_1$ such that $I$ is a subset of $J$. 

Since $E_{\bfV,\mathcal{I}_1}$ is a subset of $E_{\bfV,\mathcal{I}_2}$,
there is a natural embedding
$$\tau_{\mathcal{I}_1,\mathcal{I}_2}:E_{\bfV,\mathcal{I}_1}\rightarrow E_{\bfV,\mathcal{I}_2}.$$

\subsection{}
Let  $\mathcal{I}=(I_1,I_2,\ldots,I_n)$ be a partition of $\mathbb{R}$.
Denote by $I_i\prec I_{i+1}$ in $\mathcal{I}$, if
\begin{enumerate}
\item[(1)] $a_{i}\prec a_{i}+\epsilon$, when
$I_i=|a_{i-1},a_{i}], I_{i+1}=(a_i,a_{i+1}|$,
\item[(2)]  $a_{i}-\epsilon\prec a_{i}$, when
$I_i=|a_{i-1},a_{i}), I_{i+1}=[a_i,a_{i+1}|$,
\end{enumerate}
and by $I_{i+1}\prec I_{i}$ in $\mathcal{I}$, if
\begin{enumerate}
\item[(1)] $a_{i}+\epsilon\prec a_{i}$, when
$I_i=|a_{i-1},a_{i}], I_{i+1}=(a_i,a_{i+1}|$,
\item[(2)]  $a_{i}\prec a_{i}-\epsilon$, when
$I_i=|a_{i-1},a_{i}), I_{i+1}=[a_i,a_{i+1}|$,
\end{enumerate}
where $\epsilon$ is a real number small enough.

Consider a finite quiver $Q_{\mathcal{I}}=(Q_{\mathcal{I},0},Q_{\mathcal{I},1})$ of type $A$ associated to the   continuous quiver $A_{\mathbb{R}}$ and the partition $\mathcal{I}$ of $\mathbb{R}$, where
\begin{enumerate}
\item[(1)] $Q_{\mathcal{I},0}=\{I_1,I_2,\ldots,I_n\}$ is the set of vertices; 
\item[(2)] $Q_{\mathcal{I},1}=\{h_1,h_2,\ldots,h_{n-1}\}$, where
$$h_i=\left\{
\begin{aligned}
I_i\rightarrow{I_{i+1}},  & \qquad  I_i\prec I_{i+1};\\
I_i\leftarrow{I_{i+1}}, & \qquad I_{i+1}\prec I_{i}.
\end{aligned}
\right.$$
\end{enumerate}

For any finite dimensional $\mathbb{R}$-graded ${K}$-vector space $\bfV=\oplus_{x\in\mathbb{R}}\bfV_x$,
let $\bfV_{\mathcal{I}}$ be the $Q_{\mathcal{I},0}$-graded ${K}$-vector space with $(\bfV_{\mathcal{I}})_{I_i}=\bfV_{b_i}$, where
$$b_i=\left\{
\begin{aligned}
a_{1}-1,& \qquad  i=1;\\
\frac{a_i+a_{i-1}}{2},& \qquad 1<i<n;\\
a_{n-1}+1,& \qquad i=n.
\end{aligned}
\right.$$

Define a map 
$$\sigma_{\mathcal{I}}:E_{\bfV,\mathcal{I}}\rightarrow E_{Q_{\mathcal{I}},\bfV_{\mathcal{I}}},$$
by
\begin{enumerate}
\item[(1)]  $\sigma_{\mathcal{I}}(\phi)_{I_i\rightarrow{I_{i+1}}}=\phi_{a_i+\epsilon b_{i+1}}\phi_{a_ia_i+\epsilon}\phi_{b_ia_i}$, 
when 
$I_i=|a_{i-1},a_{i}], I_{i+1}=(a_i,a_{i+1}|$ and $a_{i}\prec a_{i}+\epsilon$ ;
\item[(2)]  $\sigma_{\mathcal{I}}(\phi)_{I_i\leftarrow{I_{i+1}}}=\phi_{a_i b_{i}}\phi_{a_i+\epsilon a_i}\phi_{b_{i+1}a_i+\epsilon}$, 
when 
$I_i=|a_{i-1},a_{i}], I_{i+1}=(a_i,a_{i+1}|$ and $a_{i}+\epsilon\prec a_{i}$;
\item[(3)]  $\sigma_{\mathcal{I}}(\phi)_{I_i\rightarrow{I_{i+1}}}=\phi_{a_i b_{i+1}}\phi_{a_i-\epsilon a_i}\phi_{b_ia_i-\epsilon}$,
when
$I_i=|a_{i-1},a_{i}), I_{i+1}=[a_i,a_{i+1}|$ and $a_{i}-\epsilon\prec a_{i}$;
\item[(4)]  $\sigma_{\mathcal{I}}(\phi)_{I_i\leftarrow{I_{i+1}}}=\phi_{a_i-\epsilon b_{i}}\phi_{a_i a_i-\epsilon}\phi_{b_{i+1}a_i}$, 
when
$I_i=|a_{i-1},a_{i}), I_{i+1}=[a_i,a_{i+1}|$ and $a_{i}\prec a_{i}-\epsilon$.
\end{enumerate}

\begin{proposition}
The map $\sigma_{\mathcal{I}}$ induces an one to one correspondence between the $G_\bfV$-orbits in $E_{\bfV,\mathcal{I}}$ and the $G_{\bfV_{\mathcal{I}}}$-orbits in $E_{Q_{\mathcal{I}},\bfV_{\mathcal{I}}}$. 
\end{proposition}

\begin{proof}
Let $\phi_1,\phi_2\in E_{\bfV,\mathcal{I}}$. Since $(\bfV,\phi_1)\cong(\bfV,\phi_2)$ if and only if $(\bfV_{\mathcal{I}},\sigma_{\mathcal{I}}(\phi_1))\cong(\bfV_{\mathcal{I}},\sigma_{\mathcal{I}}(\phi_2))$, we get the desired result.

\end{proof}

\section{Geometric realization of Ringel-Hall algebra of $A_{\mathbb{R}}$}
\subsection{}
In this section, we shall recall the definitions of Ringel-Hall algebras of $A_{\mathbb{R}}$ introduced in \cite{Sala_2019,Sala_2021}.

Let $A_\mathbb{R}$ be a continuous quiver.
Fix a dimension function $\nu$ and a $\mathbb{R}$-graded $\mathbb{F}_q$-vector space $\bfV$ such that $\underline{\dim}\bfV=\nu$.
Denote by
$$\mathcal{F}_{G_\bfV}(E_{\bfV})$$
the set of $G_\bfV$-invariant finitely supported $\mathbb{C}$-valued functions on $E_{\bfV}$.
Let $$\mathcal{F}_{A_{\mathbb{R}}}=\bigoplus_{\nu}\mathcal{F}_{G_\bfV}(E_{\bfV}).$$

For dimension functions $\nu,\nu',\nu''$ such that $\nu=\nu'+\nu''$, fix $\mathbb{R}$-graded $\mathbb{F}_q$-vector spaces $\bfV$, $\bfV'$, $\bfV''$ with $\underline{\dim}\bfV=\nu$, $\underline{\dim}\bfV'=\nu'$ and $\underline{\dim}\bfV''=\nu''$.

Consider the following diagram
$$
\xymatrix{\tilde{E}_{\bfV'}\times \tilde{E}_{\bfV''}&\tilde{E}'\ar[l]_-{q_1}\ar[r]^-{q_2}&\tilde{E}''\ar[r]^-{q_3}&\tilde{E}_{\bfV}},
$$
where
\begin{enumerate}
  \item[(1)]$\tilde{E}''=\{(\phi,\bfW)\}$, where $\phi\in \tilde{E}_{\bfV}$ and $\bfW$ is a $\phi$-stable subspace of $\bfV$ with dimension function $\nu''$;
  \item[(2)]$\tilde{E}'=\{(\phi,\bfW,R'',R')\}$, where $(\phi,\bfW)\in \tilde{E}''$, $R'':\bfV''\simeq \bfW$ and $R':\bfV'\simeq{\bfV}/\bfW$;
  \item[(3)]$q_1(\phi,\bfW,R'',R')=(\phi',\phi'')$, where $\phi'=(\phi'_{xy})$ and $\phi''=(\phi''_{xy})$ are induced by the following commutative diagrams
      $$\xymatrix{\bfV'_{x}\ar[r]^-{\phi'_{xy}}\ar[d]^-{R'_{x}}&\bfV'_{y}\ar[d]^-{R'_{y}}\\
      (\bfV/\bfW)_{x}\ar[r]^-{\phi_{xy}}&(\bfV/\bfW)_{y}}$$ and
      $$\xymatrix{\bfV''_{x}\ar[r]^-{\phi''_{xy}}\ar[d]^-{R''_{x}}&\bfV''_{y}\ar[d]^-{R''_{y}}\\
      \bfW_{x}\ar[r]^-{\phi_{xy}}&\bfW_{y}}$$
      for all $x\prec y$;
  \item[(4)]$q_2(\phi,\bfW,R'',R')=(\phi,\bfW)$;
  \item[(5)]$q_3(\phi,\bfW)=\phi$.
\end{enumerate}

The group $G_{\bfV}$ acts on $\tilde{E}''$ by $g.(\phi, \mathbf{W})=(g\phi g^{-1}, g\mathbf{W})$ for any $g\in G_{\bfV}$.
The groups $G_{\bfV'}\times G_{\bfV''}$ and $G_{\bfV}$ act on $\tilde{E}'$ by $(g_1, g_2).(\phi, \mathbf{W}, R'', R')=(\phi, \mathbf{W}, g_2R'', g_1R')$ and $g.(\phi, \mathbf{W}, R'', R')=(g\phi g^{-1}, g\mathbf{W}, R''g^{-1}, R'g^{-1})$ for any $(g_1, g_2)\in G_{\bfV'}\times G_{\bfV''}$ and $g\in G_{\bfV}$.

Since $\phi\in E_{\bfV}$ implies that $\phi'\in E_{\bfV'}$ and $\phi''\in E_{\bfV''}$, we have the following diagram
\begin{equation*}
\xymatrix{E_{\bfV'}\times E_{\bfV''}&E'\ar[l]_-{q_1}\ar[r]^-{q_2}&E''\ar[r]^-{q_3}&E_{\bfV}},
\end{equation*}
where
\begin{enumerate}
  \item[(1)]$E''=\{(\phi,\bfW)\in\tilde{E}''|\phi\in E_{\bfV}\}$;
  \item[(2)]$E'=\{(\phi,\bfW,R'',R')\in\tilde{E}'|(\phi,\bfW)\in{E}''\}$;
  \item[(3)]the restrictions of $q_1$ and $q_2$ on $E'$ are still denoted by $q_1$ and $q_2$, respectively;
  \item[(4)]the restriction of $q_3$ on $E''$ is still denoted by $q_3$.
\end{enumerate}

Since $G_{\bfV}(E'')=E''$, there is a $G_{\bfV}$-action on $E''$.
Similarly, there are $G_{\bfV'}\times G_{\bfV''}$ and $G_{\bfV}$-actions on $E'$.
The map $q_1$ induces a map
$$q^\ast_1:\mathcal{F}_{G_{\bfV'}\times G_{\bfV''}}(E_{\bfV'}\times E_{\bfV''})\rightarrow\mathcal{F}_{G_\bfV\times G_{\bfV'}\times G_{\bfV''}}(E').$$
Since the map $q_2$ is a principal $G_{\bfV'}\times G_{\bfV''}$-bundle, it induces a bijection $$q^\ast_2:\mathcal{F}_{G_\bfV}(E'')\rightarrow\mathcal{F}_{G_\bfV\times G_{\bfV'}\times G_{\bfV''}}(E').$$ 
Since $p_3^{-1}(\phi)$ is a finite set for any $\phi\in E_{\bfV}$, it induces a map
$${p_3}_{!}:\mathcal{F}_{G_\bfV}(E'')\rightarrow\mathcal{F}_{G_\bfV}(E_{\bfV}).$$
Then we can define the induction map
\begin{eqnarray*}
\ast:\mathcal{F}_{G_{\bfV'}}(E_{\bfV'})\times\mathcal{F}_{G_{\bfV''}}(E_{\bfV''})&\rightarrow&\mathcal{F}_{G_\bfV}(E_\bfV)\\
(f',f'')&\mapsto&f'\ast f''={p_3}_{!}(p_2^{\ast})^{-1}p_1^{\ast}(f'\otimes f'').
\end{eqnarray*}

\begin{proposition}[\cite{Sala_2019,Sala_2021}]
The induction maps for various dimension functions $\nu$, $\nu'$, $\nu''$ induce an associative multiplication structure $$\ast:\mathcal{F}_{A_{\mathbb{R}}}\times \mathcal{F}_{A_{\mathbb{R}}}\rightarrow \mathcal{F}_{A_{\mathbb{R}}}$$ sending $(f',f'')$ to $f'\ast f''$.
\end{proposition}

The $\mathbb{C}$-algebra $(\mathcal{F}_{A_{\mathbb{R}}},\ast)$ is called the Ringel-Hall algebra of $A_{\mathbb{R}}$.

For any partition $\mathcal{I}$ of $\mathbb{R}$,
let
$$\mathcal{F}_{G_{\bfV}}(E_{\bfV,\mathcal{I}})$$
be the set of $G_\bfV$-invariant finitely supported $\mathbb{C}$-valued functions on $E_{\bfV,\mathcal{I}}$ and $$\mathcal{F}_{A_{\mathbb{R},\mathcal{I}}}=\bigoplus_{\nu}\mathcal{F}_{G_{\bfV}}(E_{\bfV,\mathcal{I}}).$$

\begin{lemma}\label{subalg_I}
The subset $\mathcal{F}_{A_{\mathbb{R},\mathcal{I}}}$ is a subalgebra of $\mathcal{F}_{A_\mathbb{R}}$.
\end{lemma}

\begin{proof}
Note that $p_1(\phi,\bfW,R'',R')=(\phi',\phi'')$ and $(\phi',\phi'')\in E_{\bfV',\mathcal{I}}\times E_{\bfV'',\mathcal{I}}$ implies that $\phi\in E_{\bfV,\mathcal{I}}$.
Hence,
for any $f\in\mathcal{F}_{G_{\bfV'}}(E_{\bfV',\mathcal{I}})$ and $g\in\mathcal{F}_{G_{\bfV''}}(E_{\bfV'',\mathcal{I}})$, we have $f\ast g\in\mathcal{F}_{G_{\bfV}}(E_{\bfV,\mathcal{I}})$.

\end{proof}

Let  $\mathcal{I}_1$ and $\mathcal{I}_2$ be two partitions of $\mathbb{R}$ such that  $\mathcal{I}_2$ is a refinement of $\mathcal{I}_1$.
The embedding $$\tau_{\mathcal{I}_1,\mathcal{I}_2}:E_{\bfV,\mathcal{I}_1}\rightarrow E_{\bfV,\mathcal{I}_2}$$
induces a map
$$(\tau_{\mathcal{I}_1,\mathcal{I}_2})_!:\mathcal{F}_{G_{\bfV}}(E_{\bfV,\mathcal{I}_1})\rightarrow \mathcal{F}_{G_{\bfV}}(E_{\bfV,\mathcal{I}_2}).$$

\begin{corollary}\label{cor_sub_alg}
The algebra $\mathcal{F}_{A_{\mathbb{R},\mathcal{I}_1}}$ is a subalgebra of $\mathcal{F}_{A_{\mathbb{R},\mathcal{I}_2}}$ and $(\tau_{\mathcal{I}_1,\mathcal{I}_2})_!$ is an embedding of algebras.
\end{corollary}

\subsection{}

Let $Q_{\mathcal{I}}=(Q_{\mathcal{I},0},Q_{\mathcal{I},1})$ be the finite quiver associated to the continuous quiver $A_{\mathbb{R}}$ and a partition $\mathcal{I}$ and ${\bfV}_{\mathcal{I}}$ be the $Q_{\mathcal{I},0}$-graded vector space induced by $\bfV$.

The map 
$$\sigma_{\mathcal{I}}:E_{\bfV,\mathcal{I}}\rightarrow E_{Q_{\mathcal{I}},\bfV_{\mathcal{I}}}$$
induces a map
$$\sigma^\ast_{\mathcal{I}}:\mathcal{F}_{G_{\bfV_{\mathcal{I}}}}(E_{Q_{\mathcal{I}},\bfV_{\mathcal{I}}})
\rightarrow
\mathcal{F}_{G_{\bfV}}(E_{\bfV,\mathcal{I}}).
$$

\begin{lemma}\label{iso_I_1}
The map $\sigma^\ast_{\mathcal{I}}$ is an isomorphism of algebras.
\end{lemma}

\begin{proof}
Consider the following commutative diagram
\begin{equation*}
\xymatrix{E_{\bfV',\mathcal{I}}\times E_{\bfV'',\mathcal{I}}\ar[d]^-{\sigma_{\mathcal{I}}\times\sigma_{\mathcal{I}}}&E_1'\ar[d]^-{\gamma_1}\ar[l]_-{q_1}\ar[r]^-{q_2}&E_1''\ar[d]^-{\delta_1}\ar[r]^-{q_3}&E_{\bfV,\mathcal{I}}\ar[d]^-{id}\\
E_{\bfV'_\mathcal{I}}\times E_{\bfV''_\mathcal{I}}\ar[d]^-{id\times id}&X\ar[d]^-{\gamma_2}\ar[l]_-{\alpha_1}\ar[r]^-{\alpha_2}&Y\ar[d]^-{\delta_2}\ar[r]^-{\alpha_3}&E_{\bfV,\mathcal{I}}\ar[d]^-{\sigma_{\mathcal{I}}}\\
E_{\bfV'_\mathcal{I}}\times E_{\bfV''_\mathcal{I}}&E_2'\ar[l]_-{p_1}\ar[r]^-{p_2}&E_2''\ar[r]^-{p_3}&E_{\bfV_\mathcal{I}}.}
\end{equation*}
Here,
\begin{enumerate}
  \item[(1)]$E_1''=\{(\phi,\bfW)\}$, where $\phi\in E_{\bfV,\mathcal{I}}$ and $\bfW$ is a $\phi$-stable subspace of $\bfV$ with dimension function $\nu''$;
  \item[(2)]$E_1'=\{(\phi,\bfW,R'',R')\}$, where $(\phi,\bfW)\in E_1''$, $R'':\bfV''\simeq \bfW$ and $R':\bfV'\simeq{\bfV}/\bfW$;
  \item[(3)]$E_2''=\{(\phi,\bbW)\}$, where $\phi\in E_{\bfV_\mathcal{I}}$ and $\bbW$ is a $\phi$-stable subspace of $\bfV_\mathcal{I}$ with dimension vector $\nu''_\mathcal{I}$;
  \item[(4)]$E_2'=\{(\phi,\bbW,R'',R')\}$, where $(\phi,\bbW)\in E_1''$, $R'':\bfV''_\mathcal{I}\simeq\bbW$ and $R':\bfV'_\mathcal{I}\simeq{\bfV_\mathcal{I}}/\bbW$;
  \item[(5)]$Y=\{(\phi,\bbW)\}$, where $\phi\in E_{\bfV,\mathcal{I}}$ and $\bbW$ is an $\sigma_{\mathcal{I}}(\phi)$-stable subspace of $\bfV_\mathcal{I}$ with dimension function $\nu''_\mathcal{I}$;
  \item[(6)]$X=\{(\phi,\bbW,R'',R')\}$, where $(\phi,\bbW)\in Y$, $R'':\bfV''_\mathcal{I}\simeq\bbW$ and $R':\bfV'_\mathcal{I}\simeq{\bfV_\mathcal{I}}/\bbW$;
  \item[(7)]$\delta_1(\phi,\bfW)=(\phi,\bfW_{\mathcal{I}})$, which is a bijection;
  \item[(8)]$\gamma_1(\phi,\bfW,R'',R')=(\phi,\bfW_{\mathcal{I}},R''_{\mathcal{I}},R'_{\mathcal{I}})$, where $R_{\mathcal{I}}'':\bfV''_\mathcal{I}\simeq\bfW_{\mathcal{I}}$ and $R_{\mathcal{I}}':\bfV'_\mathcal{I}\simeq{\bfV_\mathcal{I}}/\bfW_{\mathcal{I}}$ are indeced by $R''$ and $R'$;
  \item[(9)]$\delta_2(\phi,\bbW)=(\sigma_{\mathcal{I}}(\phi),\bbW)$;
  \item[(10)]$\gamma_2(\phi,\bbW,R'',R')=(\sigma_{\mathcal{I}}(\phi),\bbW,R'',R')$;
  \item[(11)]$\alpha_1=p_1\gamma_2$;
  \item[(12)]$\alpha_2(\phi,\bbW,R'',R')=(\phi,\bbW)$;
  \item[(13)]$\alpha_3(\phi,\bbW)=\phi$.
\end{enumerate}

Now, for any $f\in\mathcal{F}_{G_{\bfV'_\mathcal{I}}}(E_{\bfV'_\mathcal{I}})$
and $g\in\mathcal{F}_{G_{\bfV''_\mathcal{I}}}(E_{\bfV''_\mathcal{I}})$, let
$A=\sigma^\ast_{\mathcal{I}}(f\ast g)$.
By definition, we have $$A=\sigma^\ast_{\mathcal{I}}{p_3}_{!}({p_2^\ast})^{-1}p_1^{\ast}(f\otimes g).$$
Since the commutative diagram
$$
\xymatrix{
Y\ar[d]^-{\delta_2}\ar[r]^-{\alpha_3}&E_{\bfV,\mathcal{I}}\ar[d]^-{\sigma_{\mathcal{I}}}\\
E_2''\ar[r]^-{p_3}&E_{\bfV_\mathcal{I}}}
$$
is a fiber product, we have
$$A={\alpha_3}_{!}\delta_2^\ast({p_2^\ast})^{-1}p_1^{\ast}(f\otimes g).$$
Then,
$$
A={\alpha_3}_{!}({\alpha_2^\ast})^{-1}\gamma_2^\ast p_1^{\ast}(f\otimes g)
={\alpha_3}_{!}({\alpha_2^\ast})^{-1}\alpha_1^{\ast}(f\otimes g).
$$
Since $\delta_1$ is a bijection, we have $\delta^\ast_1=(\delta_{1!})^{-1}$ and
$$
A
={q_3}_{!}\delta_1^\ast({\alpha_2^\ast})^{-1}\alpha_1^{\ast}(f\otimes g).
$$
Then, 
$$
A
={q_3}_{!}({q_2^\ast})^{-1}\gamma_1^\ast \alpha_1^{\ast}(f\otimes g)
={q_3}_{!}({q_2^\ast})^{-1}q_1^{\ast}(\sigma^\ast_{\mathcal{I}}(f)\otimes \sigma^\ast_{\mathcal{I}}(g))
=\sigma^\ast_{\mathcal{I}}(f)\ast\sigma^\ast_{\mathcal{I}}(g).
$$

Hence, the map $\sigma^\ast_{\mathcal{I}}$ is an isomorphism of algebras.

\end{proof}

Let $\bbV_{\mathcal{I}}$ be a $Q_{\mathcal{I},0}$-graded $K$-vector space such that $\bbV_{\mathcal{I}}^F={\bfV}_{\mathcal{I}}$. Then, we have $E^F_{\bbV_{\mathcal{I}}}=E_{\bfV_\mathcal{I}}$ and $G^F_{\bbV_{\mathcal{I}}}=G_{\bfV_\mathcal{I}}$.
Consider the trace map
$$\chi_{\mathcal{I}}:K_{G_{{\bbV_{\mathcal{I}}}}}^b(E_{\bbV_{\mathcal{I}}})\rightarrow \mathcal{F}_{G_{\bfV_\mathcal{I}}}(E_{\bfV_\mathcal{I}})\rightarrow\mathcal{F}_{G_\bfV}(E_{\bfV,\mathcal{I}})\rightarrow\mathcal{F}_{G_\bfV}(E_\bfV)$$
sending $[\mathcal{L}]$ to $\sigma^\ast_{\mathcal{I}}(\chi^{F}_\mathcal{L})$.

\begin{proposition}
Considering all dimension vectors, the map $$\chi_{\mathcal{I}}: K_{Q_\mathcal{I}}\rightarrow \mathcal{F}_{A_{\mathbb{R}}}$$ is a homomorphism of $\mathbb{Z}$-algebras.
\end{proposition}

\begin{proof}
This result is direct by Lemma \ref{iso_I_1} and \ref{subalg_I}.

\end{proof}

The homomorphism $\chi_{\mathcal{I}}$ of $\mathbb{Z}$-algebras
 induces a homomorphsim of $\mathbb{C}$-algebras $$\chi^{\mathbb{C}}_{\mathcal{I}}: K_{Q_\mathcal{I}}\otimes_{\mathbb{Z}[v,v^{-1}]}\mathbb{C}\rightarrow \mathcal{F}_{A_{\mathbb{R}}}.$$

\subsection{}\label{sec_4.3}

Let  $\mathcal{I}_1$ and $\mathcal{I}_2$ be two partitions of $\mathbb{R}$ such that the partition $\mathcal{I}_2$ is a refinement of $\mathcal{I}_1$.
Assume that $$\mathcal{I}_1=(I_1,I_2,\ldots,I_n)$$ and 
$$\mathcal{I}_2=(I_{11},I_{12},\ldots,I_{1s_1},\ldots,I_{n1},I_{n2},\ldots,I_{ns_n})$$
such that $I_i=\bigcup_{j=1}^{n_i}I_{ij}$.

Let $\hat{E}$ be the subset of $E_{{\bbV}_{\mathcal{I}_2}}$ consisting of $\phi$ such that $\phi_{h_{ij}}$ are isomorphisms for all $1\leq j<n_i$.
For any $\phi\in\hat{E}$ and $1\leq j<n_i$, let $\phi_{ij}=\phi_{h_{ij}}$ if $I_{ij}\prec I_{ij+1}$ and 
$\phi_{ij}=\phi^{-1}_{h_{ij}}$ if $I_{ij+1}\prec I_{ij}$.

Consider the following correspondence
\begin{equation}\label{corres_1}
  \xymatrix{E_{{\bbV}_{\mathcal{I}_1}}&\hat{E}\ar[l]_-{r_1}\ar[r]^-{r_2}&E_{{\bbV}_{\mathcal{I}_2}}},  
\end{equation}
where
$${r_1(\phi)}_{h_i}=\left\{
\begin{aligned}
\phi_{h_{in_i}}\phi_{in_i-1}\ldots\phi_{i1},& \qquad  s(h_i)=I_i;\\
\phi^{-1}_{i1}\ldots\phi^{-1}_{in_i-1}\phi_{h_{in_i}},& \qquad s(h_i)=I_{i+1};
\end{aligned}
\right.$$
for any $\phi\in\hat{E}$ and $r_2$ is the natural embedding.

The group $G_{\bbV_{\mathcal{I}_2}}$ acts on $\hat{E}$.
Note that $G_{\bbV_{\mathcal{I}_2}}=G_{\bbV_{\mathcal{I}_1}}\times G_{\bbV_{\mathcal{I}_2-\mathcal{I}_1}}$, where $$G_{\bbV_{\mathcal{I}_2-\mathcal{I}_1}}=\prod_{i=1}^{n}\prod^{n_i}_{j=2}GL_{K}(({\bbV_{\mathcal{I}_2}})_{I_{ij}}).$$
The map $r_1$ is a principal $G_{\bbV_{\mathcal{I}_2-\mathcal{I}_1}}$-bundle.
Hence, there is an equivalence
$$r^\ast_1:\cD_{G_{\bbV_{\mathcal{I}_1}}}^b(E_{{\bbV}_{\mathcal{I}_1}})\rightarrow\cD_{G_{\bbV_{\mathcal{I}_2}}}^b(\hat{E}).$$
The map $r_2$ is $G_{\bbV_{\mathcal{I}_2}}$-equivariant, which induces a functor
$${r_2}_!:\cD_{G_{\bbV_{\mathcal{I}_2}}}^b(\hat{E})\rightarrow\cD_{G_{\bbV_{\mathcal{I}_2}}}^b(E_{{\bbV}_{\mathcal{I}_2}}).$$

Then we can define a functor
$$\Phi_{\mathcal{I}_1,\mathcal{I}_2}={r_2}_!r_1^\ast:\cD_{G_{\bbV_{\mathcal{I}_1}}}^b(E_{{\bbV}_{\mathcal{I}_1}})\rightarrow\cD_{G_{\bbV_{\mathcal{I}_2}}}^b(E_{{\bbV}_{\mathcal{I}_2}}).$$
Consider all dimension vectors, we get
a map
$$\Phi_{\mathcal{I}_1,\mathcal{I}_2}:K_{Q_{\mathcal{I}_1}}\rightarrow K_{Q_{\mathcal{I}_2}}.$$

\begin{lemma}\label{homo_alg_two_par}
For any partitions $\mathcal{I}_1$ and $\mathcal{I}_2$ of $\mathbb{R}$ such that $\mathcal{I}_2$ is a refinement of $\mathcal{I}_1$,
the map $\Phi_{\mathcal{I}_1,\mathcal{I}_2}:K_{Q_{\mathcal{I}_1}}\rightarrow K_{Q_{\mathcal{I}_2}}$ is a homomorphism of $\mathbb{Z}[v,v^{-1}]$-algebras.
\end{lemma}

\begin{proof}
Consider the following commutative diagram
$$
\xymatrix{E_{\bbV'_{\mathcal{I}_1}}\times E_{\bbV''_{\mathcal{I}_1}}&E_1'\ar[l]_-{p_1}\ar[r]^-{p_2}&E_1''\ar[r]^-{p_3}&E_{\bbV_{\mathcal{I}_1}}\\
E_{\bbV'_{\mathcal{I}_1}}\times E_{\bbV''_{\mathcal{I}_1}}\ar[u]_-{id}&X_1\ar[u]_-{\gamma_1}\ar[l]_-{\alpha_1}\ar[r]^-{\alpha_2}&Y_1\ar[u]_-{\delta_1}\ar[r]^-{\alpha_3}&\hat{E}\ar[u]_-{r_1}\\
\hat{E}'\times \hat{E}''\ar[d]^-{r_2\times r_2}\ar[u]_-{r_1\times r_1}& X_2\ar[l]_-{\beta_1}\ar[r]^-{\beta_2}\ar[d]^-{\gamma_3}\ar[u]_-{\gamma_2} & Y_2\ar[d]^-{\delta_3}\ar[u]_-{\delta_2}\ar[r]^-{\beta_3}&\hat{E}\ar[d]^-{r_2}\ar[u]_-{id}
\\
E_{\bbV'_{\mathcal{I}_2}}\times E_{\bbV''_{\mathcal{I}_2}}&E_2'\ar[l]_-{p_1}\ar[r]^-{p_2}&E_2''\ar[r]^-{p_3}&E_{\bbV_{\mathcal{I}_2}}
.}
$$
Here,
\begin{enumerate}
  \item[(1)]$E_1''=\{(\phi,\bbW)\}$, where $\phi\in E_{\bbV_{\mathcal{I}_1}}$ and $\bbW$ is a $\phi$-stable subspace of $\bbV_{\mathcal{I}_1}$ with dimension vector $\nu''_{\mathcal{I}_1}$;
  \item[(2)]$E_1'=\{(\phi,\bbW,R'',R')\}$, where $(\phi,\bbW)\in E_1''$, $R'':\bbV''_{\mathcal{I}_1}\simeq\bbW$ and $R':\bbV'_{\mathcal{I}_1}\simeq{\bbV_{\mathcal{I}_1}}/\bbW$;
  \item[(3)]$E_2''=\{(\phi,\bbW)\}$, where $\phi\in E_{\bbV_{\mathcal{I}_2}}$ and $\bbW$ is a $\phi$-stable subspace of $\bbV_{\mathcal{I}_2}$ with dimension vector $\nu''_{\mathcal{I}_2}$;
  \item[(4)]$E_2'=\{(\phi,\bbW,R'',R')\}$, where $(\phi,\bbW)\in E_2''$, $R'':\bbV''_{\mathcal{I}_2}\simeq\bbW$ and $R':\bbV'_{\mathcal{I}_2}\simeq{\bbV_{\mathcal{I}_2}}/\bbW$;
  \item[(5)]$\hat{E}$ is the subset of $E_{{\bbV}_{\mathcal{I}_2}}$ consisting of $\phi$ such that $\phi_{h_{ij}}$ are isomorphism for all $1\leq j<n_i$;
  \item[(6)]$\hat{E}'$ is the subset of $E_{\bbV'_{\mathcal{I}_2}}$ consisting of $\phi$ such that $\phi_{h_{ij}}$ are isomorphism for all $1\leq j<n_i$;
  \item[(7)]$\hat{E}''$ is the subset of $E_{\bbV''_{\mathcal{I}_2}}$ consisting of $\phi$ such that $\phi_{h_{ij}}$ are isomorphism for all $1\leq j<n_i$;
  \item[(8)]$Y_2=\{(\phi,\bbW)\}$, where $\phi\in\hat{E}$ and $\bbW$ is a $\phi$-stable subspace of $\bbV_{\mathcal{I}_2}$ with dimension function $\nu''_{\mathcal{I}_2}$;
  \item[(9)]$X_2=\{(\phi,\bbW,R'',R')\}$, where $(\phi,\bbW)\in Y_2$, $R'':\bbV''_{\mathcal{I}_2}\simeq \bbW$ and $R':\bbV'_{\mathcal{I}_2}\simeq{\bbV_{\mathcal{I}_2}}/\bbW$;
  \item[(10)]$Y_1=\{(\phi,\bbW)\}$, where $\phi\in\hat{E}$ and $\bbW$ is a $\phi$-stable subspace of $\bbV_{\mathcal{I}_1}$ with dimension function $\nu''_{\mathcal{I}_1}$;
  \item[(11)]$X_1=\{(\phi,W,R'',R')\}$, where $(\phi,\bbW)\in Y_1$, $R'':\bbV''_{\mathcal{I}_1}\simeq \bbW$ and $R':\bbV'_{\mathcal{I}_1}\simeq{\bbV_{\mathcal{I}_1}}/\bbW$;
  \item[(12)]$\delta_1(\phi,\bbW)=(r_1(\phi),\bbW)$;
  \item[(13)]$\gamma_1(\phi,\bbW,R'',R')=(r_1(\phi),\bbW,R'',R')$;
  \item[(14)]$\delta_2(\phi,\bbW)=(\phi,\bbW_{\mathcal{I}_1})$, where $\bbW_{\mathcal{I}_1}$ is the restriction of $\bbW$ on $\mathcal{I}_1$;
  \item[(15)]$\gamma_2(\phi,\bbW,R'',R')=(\phi,\bbW_{\mathcal{I}_1},R''_{\mathcal{I}_1},R'_{\mathcal{I}_1})$, where $R''_{\mathcal{I}_1}$ and $R'_{\mathcal{I}_1}$ are the restrictions of $R''$ and $R'$ on $\mathcal{I}_1$;
  \item[(16)]$\delta_3(\phi,\bbW)=(\phi,\bbW)$;
  \item[(17)]$\gamma_3(\phi,\bbW,R'',R')=(\phi,\bbW,R'',R')$;
  \item[(18)]$\alpha_1=p_1\gamma_1$;
  \item[(19)]$\alpha_2(\phi,\bbW,R'',R')=(\phi,\bbW)$;
  \item[(20)]$\alpha_3(\phi,\bbW)=\phi$.
  \item[(21)]$\beta_1=p_1|_{X_2}$;
  \item[(22)]$\beta_2(\phi,\bbW,R'',R')=(\phi,\bbW)$;
  \item[(23)]$\beta_3(\phi,\bbW)=\phi$.
\end{enumerate}

Now, for any $\mathcal{L}_1\in\cD_{G_{\bbV'_{\mathcal{I}_1}}}^b(E_{\bbV'_{\mathcal{I}_1}})$
and $\mathcal{L}_2\in\cD_{G_{\bbV''_{\mathcal{I}_1}}}^b(E_{\bbV''_{\mathcal{I}_1}})$, let
$$A=\Phi_{\mathcal{I}_1,\mathcal{I}_2}(\mathcal{L}_1\ast\mathcal{L}_2).$$
By definition, we have
$$A\cong{r_2}_! r_1^\ast{p_3}_{!}{p_2}_{\flat}p_1^{\ast}(\mathcal{L}_1\boxtimes\mathcal{L}_2).$$
Since the commutative diagram$$
\xymatrix{
 Y_1\ar[d]^-{\delta_1}\ar[r]^-{\alpha_3}&\hat{E}\ar[d]^-{r_1}\\
E_1''\ar[r]^-{p_3}&E_{\bbV_{\mathcal{I}_1}}}
$$
is a fiber product, we have
$$A\cong{r_2}_! {\alpha_3}_{!}\delta_1^\ast{p_2}_{\flat}p_1^{\ast}(\mathcal{L}_1\boxtimes\mathcal{L}_2).$$
Since $\delta_2$ is an isomorphism and $\beta_3=\alpha_3\delta_2$, we have
$$A\cong{r_2}_! {\beta_3}_{!}\delta_2^\ast\delta_1^\ast{p_2}_{\flat}p_1^{\ast}(\mathcal{L}_1\boxtimes\mathcal{L}_2).$$
Then we have
\begin{eqnarray*}
A&\cong&{p_3}_! {\delta_3}_{!}\delta_2^\ast\delta_1^\ast{p_2}_{\flat}p_1^{\ast}(\mathcal{L}_1\boxtimes\mathcal{L}_2)\\
&\cong&{p_3}_! {\delta_3}_{!}\delta_2^\ast{\alpha_2}_{\flat}\gamma_1^\ast p_1^{\ast}(\mathcal{L}_1\boxtimes\mathcal{L}_2)\\
&\cong&{p_3}_! {\delta_3}_{!}{\beta_2}_{\flat}\gamma_2^\ast\gamma_1^\ast p_1^{\ast}(\mathcal{L}_1\boxtimes\mathcal{L}_2)\\
&\cong&{p_3}_! {p_2}_{\flat}{\gamma_3}_{!}\gamma_2^\ast \gamma_1^\ast p_1^{\ast}(\mathcal{L}_1\boxtimes\mathcal{L}_2)\\
&\cong&{p_3}_! {p_2}_{\flat}{\gamma_3}_{!}\gamma_2^\ast \alpha_1^\ast(\mathcal{L}_1\boxtimes\mathcal{L}_2)\\
&\cong&{p_3}_! {p_2}_{\flat}{\gamma_3}_{!} \beta_1^\ast(r_1^\ast(\mathcal{L}_1)\boxtimes r_1^\ast(\mathcal{L}_2)).
\end{eqnarray*}
Since the commutative diagram
$$
\xymatrix{X_2\ar[r]^-{\beta_1}\ar[d]^-{\gamma_3}&\hat{E}'\times\hat{E}''\ar[d]^-{r_2\times r_2}\\
E'_2\ar[r]^-{p_1}&E_{\bbV'_{\mathcal{I}_2}}\times E_{\bbV''_{\mathcal{I}_2}}}
$$
is a fiber product, we have
\begin{eqnarray*}
A&\cong&{p_3}_! {p_2}_{\flat}p_1^\ast({r_2}_{!}r_1^\ast(\mathcal{L}_1)\boxtimes {r_2}_{!}r_1^\ast(\mathcal{L}_2))\\
&\cong&{p_3}_{!}{p_2}_{\flat}p_1^\ast(\Phi_{\mathcal{I}_1,\mathcal{I}_2}(\mathcal{L}_1)\boxtimes\Phi_{\mathcal{I}_1,\mathcal{I}_2}(\mathcal{L}_2))\\
&\cong&
\Phi_{\mathcal{I}_1,\mathcal{I}_2}(\mathcal{L}_1)\ast\Phi_{\mathcal{I}_1,\mathcal{I}_2}(\mathcal{L}_2).\end{eqnarray*}

Hence, the map $\Phi_{\mathcal{I}_1,\mathcal{I}_2}$ is a homomorphism of algebras.

\end{proof}

Note that $$\Phi_{\mathcal{I}_2,\mathcal{I}_3}\Phi_{\mathcal{I}_1,\mathcal{I}_2}=\Phi_{\mathcal{I}_1,\mathcal{I}_3}$$ for partitions $\mathcal{I}_1$, $\mathcal{I}_2$ and $\mathcal{I}_3$ of $\mathbb{R}$ such that $\mathcal{I}_2$ is a refinement of $\mathcal{I}_1$ and $\mathcal{I}_3$ is a refinement of $\mathcal{I}_2$.
Hence, 
$$(K_{Q_\mathcal{I}},\Phi_{\mathcal{I}_1,\mathcal{I}_2})$$ is a direct system.
Let $$(K_{A_{\mathbb{R}}},\Phi_{\mathcal{I}})$$ be the direct limit of this direct system, where $K_{A_{\mathbb{R}}}$ is an $\mathbb{Z}[v,v^{-1}]$-algebra and 
$$\Phi_{\mathcal{I}}: K_{Q_\mathcal{I}}\rightarrow K_{A_{\mathbb{R}}}$$
is homomorphism of algebras.

Similarly, consider the following correspondence over $\mathbb{F}_q$
\begin{equation}\label{corres_2}
 \xymatrix{E_{{\bfV}_{\mathcal{I}_1}}&\hat{E}\ar[l]_-{r_1}\ar[r]^-{r_2}&E_{{\bfV}_{\mathcal{I}_2}}},   
\end{equation}
and we can define a map $$\Phi^F_{\mathcal{I}_1,\mathcal{I}_2}={r_2}_! r_1^\ast:\mathcal{F}_{G_{{\bfV}_{\mathcal{I}_1}}}(E_{{\bfV}_{\mathcal{I}_1}})\rightarrow\mathcal{F}_{G_{{\bfV}_{\mathcal{I}_2}}}(E_{{\bfV}_{\mathcal{I}_2}}).$$

Consider all dimension vectors, we get a
map
$$\Phi^F_{\mathcal{I}_1,\mathcal{I}_2}:\mathcal{F}_{Q_{\mathcal{I}_1}}\rightarrow \mathcal{F}_{Q_{\mathcal{I}_2}}.$$

\begin{lemma}\label{prop_homo_alg_fun}
For any partitions $\mathcal{I}_1$ and $\mathcal{I}_2$ of $\mathbb{R}$ such that $\mathcal{I}_2$ is a refinement of $\mathcal{I}_1$,
the map $\Phi^F_{\mathcal{I}_1,\mathcal{I}_2}$ is a homomorphism of algebras such that the follwoing diagram is commutative
$$\xymatrix{
K_{Q_{\mathcal{I}_1}}\ar[r]^-{\Phi_{\mathcal{I}_1,\mathcal{I}_2}}\ar[d]_-{\chi^F_{Q_{\mathcal{I}_1}}}&K_{Q_{\mathcal{I}_2}}\ar[d]^-{\chi^F_{Q_{\mathcal{I}_2}}}\\
\mathcal{F}_{Q_{\mathcal{I}_1}}\ar[r]^-{\Phi^F_{\mathcal{I}_1,\mathcal{I}_2}}&\mathcal{F}_{Q_{\mathcal{I}_2}}.
}$$
\end{lemma}

\begin{proof}
Similarly to the proof of Lemma \ref{homo_alg_two_par}, we can show that $\Phi^F_{\mathcal{I}_1,\mathcal{I}_2}$ is a homomorphism of algebras.
By the properties of sheaf-function correspondence (\cite{Kiehl_Weissauer_Weil_conjectures_perverse_sheaves_and_l'adic_Fourier_transform}), we get this commutative diagram.

\end{proof}

On the relations between $\Phi^F_{\mathcal{I}_1,\mathcal{I}_2}$ and $(\tau_{\mathcal{I}_1,\mathcal{I}_2})_!$, we have the following lemma.

\begin{lemma}\label{prop_homo_alg_fun_1}
For any partitions $\mathcal{I}_1$ and $\mathcal{I}_2$ of $\mathbb{R}$ such that $\mathcal{I}_2$ is a refinement of $\mathcal{I}_1$, we have the following commutative diagram
$$\xymatrix{
\mathcal{F}_{Q_{\mathcal{I}_1}}\ar[r]^-{\Phi^F_{\mathcal{I}_1,\mathcal{I}_2}}\ar[d]^-{\sigma_{\mathcal{I}_1}^\ast}&\mathcal{F}_{Q_{\mathcal{I}_2}}\ar[d]^-{\sigma_{\mathcal{I}_2}^\ast}
\\
\mathcal{F}_{A_{\mathbb{R}},{\mathcal{I}_1}}\ar[r]^-{(\tau_{\mathcal{I}_1,\mathcal{I}_2})_!}&\mathcal{F}_{A_{\mathbb{R}},{\mathcal{I}_2}}
.}$$
\end{lemma}

\begin{proof}
Let $V=(\bfV,\phi)$ be a representation of $A_{\mathbb{R}}$ such that $\phi\in E_{\bfV,\mathcal{I}_1}$.
Let $V_1=(\bfV_{\mathcal{I}_1},\sigma_{\mathcal{I}_1}(\phi))$ and $V_2=(\bfV_{\mathcal{I}_2},\sigma_{\mathcal{I}_2}(\phi))$.
Let $V_3=(\bfV_{\mathcal{I}_2},\psi)$ be a representation of $Q_{\mathcal{I}_2}$ such that $r_1(\psi)=\sigma_{\mathcal{I}_1}(\phi)$.

Then, we have
$$(\sigma_{\mathcal{I}_2}^\ast)^{-1}(\tau_{\mathcal{I}_1,\mathcal{I}_2})_!(1_{\mathcal{O}_V})=(\sigma_{\mathcal{I}_2}^\ast)^{-1}(1_{\mathcal{O}_V})=1_{\mathcal{O}_{V_2}}$$
and
$$\Phi^F_{\mathcal{I}_1,\mathcal{I}_2}(\sigma_{\mathcal{I}_1}^\ast)^{-1}(1_{\mathcal{O}_V})=\Phi^F_{\mathcal{I}_1,\mathcal{I}_2}(1_{\mathcal{O}_{V_1}})=1_{\mathcal{O}_{V_3}}.$$
Since $V_2\cong V_3$, we have $1_{\mathcal{O}_{V_2}}=1_{\mathcal{O}_{V_3}}$
and 
$$\Phi^F_{\mathcal{I}_1,\mathcal{I}_2}(\sigma_{\mathcal{I}_1}^\ast)^{-1}(1_{\mathcal{O}_V})=(\sigma_{\mathcal{I}_2}^\ast)^{-1}(\tau_{\mathcal{I}_1,\mathcal{I}_2})_!(1_{\mathcal{O}_V}).$$

Hence, we have
$$\Phi^F_{\mathcal{I}_1,\mathcal{I}_2}(\sigma_{\mathcal{I}_1}^\ast)^{-1}=(\sigma_{\mathcal{I}_2}^\ast)^{-1}(\tau_{\mathcal{I}_1,\mathcal{I}_2})_!$$
and
$$\sigma_{\mathcal{I}_2}^\ast\Phi^F_{\mathcal{I}_1,\mathcal{I}_2}=(\tau_{\mathcal{I}_1,\mathcal{I}_2})_!\sigma_{\mathcal{I}_1}^\ast.$$
   
\end{proof}

\begin{proposition}\label{prop_1}
For any partitions $\mathcal{I}_1$ and $\mathcal{I}_2$ of $\mathbb{R}$ such that $\mathcal{I}_2$ is a refinement of $\mathcal{I}_1$, we have the following commutative diagram
$$\xymatrix{K_{Q_{\mathcal{I}_1}}\ar[rr]^-{\Phi_{\mathcal{I}_1,\mathcal{I}_2}}\ar[dr]_-{\chi_{\mathcal{I}_1}}&&K_{Q_{\mathcal{I}_2}}\ar[dl]^-{\chi_{\mathcal{I}_2}}\\&\mathcal{F}_{A_{\mathbb{R}}}.}$$
\end{proposition}
\begin{proof}
By Corollary \ref{cor_sub_alg}, Lemmas \ref{prop_homo_alg_fun} and \ref{prop_homo_alg_fun_1}, we get the desired commutative diagram.

\end{proof}

\begin{theorem}
There is a  homomorphism of $\mathbb{Z}$-algebras $$\chi_{{A_{\mathbb{R}}}}:K_{A_{\mathbb{R}}}\rightarrow\mathcal{F}_{A_{\mathbb{R}}}.$$
\end{theorem}

\begin{proof}
Since $K_{A_{\mathbb{R}}}$ is the direct limit of
$(K_{Q_\mathcal{I}},\Phi_{\mathcal{I}_1,\mathcal{I}_2})$, there is a  homomorphism of $\mathbb{Z}$-algebras
$$\chi_{{A_{\mathbb{R}}}}:K_{A_{\mathbb{R}}}\rightarrow\mathcal{F}_{A_{\mathbb{R}}}$$
by Proposition \ref{prop_1}.

\end{proof}

The homomorphism $\chi_{{A_{\mathbb{R}}}}$ of $\mathbb{Z}$-algebras
induces a homomorphism of $\mathbb{C}$-algebras $$\chi^{\mathbb{C}}_{{A_{\mathbb{R}}}}:K_{A_{\mathbb{R}}}\otimes_{\mathbb{Z}[v,v^{-1}]}\mathbb{C}\rightarrow\mathcal{F}_{A_{\mathbb{R}}}.$$ 

\begin{proposition}
The  homomorphism
$\chi^{\mathbb{C}}_{{A_{\mathbb{R}}}}$ is an epimorphism of $\mathbb{C}$-algebras. 
\end{proposition}

\begin{proof}
For any $f\in\mathcal{F}_{A_{\mathbb{R}}}$, we have $f=c_11_{\mathcal{Q}_{V_1}}+c_21_{\mathcal{Q}_{V_2}}+\cdots+c_s1_{\mathcal{Q}_{V_s}}$ with finitely generated representations $V_i$ and  real numbers $c_i$ ($1\leq i\leq s$).
For any $1\leq i\leq s$, there is a partition $\mathcal{I}_i$ of $\mathbb{R}$ such that $1_{\mathcal{Q}_{V_i}}\in\textrm{Im}(\chi_{\mathcal{I}_i})$. Then, we can construct a partition $\mathcal{I}$ of $\mathbb{R}$ such than $\mathcal{I}$ is a refinement of all $\mathcal{I}_i$. Hence, $1_{\mathcal{Q}_{V_i}}\in\textrm{Im}(\chi_{\mathcal{I}})$ and $f\in\textrm{Im}(\chi^{\mathbb{C}}_{\mathcal{I}})$.
By the constructions of $K_{A_{\mathbb{R}}}$ and $\chi^{\mathbb{C}}_{A_{\mathbb{R}}}$, we have $f\in\textrm{Im}(\chi^{\mathbb{C}}_{A_{\mathbb{R}}})$.

Hence, the $\mathbb{C}$-linear homomorphism $\chi^{\mathbb{C}}_{{A_{\mathbb{R}}}}$ is an epimorphism. 

\end{proof}

The algebra $K_{A_{\mathbb{R}}}$ is called the geometric version of Ringel-Hall algebra $\mathcal{F}_{A_{\mathbb{R}}}$ of $A_{\mathbb{R}}$.

\section{Canonical basis for a completion of $K_{A_{\mathbb{R}}}$}

\subsection{}
Let $A_\mathbb{R}$ be a continuous quiver.
Fix a dimension function $\nu$ and a $\mathbb{R}$-graded $\mathbb{F}_q$-vector space $\bfV$ such that $\underline{\dim}\bfV=\nu$.
Denote by
$\bar{\mathcal{F}}_{G_\bfV}(E_{\bfV})$
the set of $G_\bfV$-invariant $\mathbb{C}$-valued functions on $E_{\bfV}$.
Let $$\bar{\mathcal{F}}_{A_{\mathbb{R}}}=\bigoplus_{\nu}\bar{\mathcal{F}}_{G_\bfV}(E_{\bfV}).$$
Then
$\mathcal{F}_{A_{\mathbb{R}}}$ is a subspace of $\bar{\mathcal{F}}_{A_{\mathbb{R}}}$ and the embedding is denoted by
$$\theta_{A_{\mathbb{R}}}: \mathcal{F}_{A_{\mathbb{R}}}\rightarrow\bar{\mathcal{F}}_{A_{\mathbb{R}}}.$$

Let  $\mathcal{I}_1$ and $\mathcal{I}_2$ be two partitions of $\mathbb{R}$ such that the partition $\mathcal{I}_2$ is a refinement of $\mathcal{I}_1$.
The embedding $$\tau_{\mathcal{I}_1,\mathcal{I}_2}:E_{\bfV,\mathcal{I}_1}\rightarrow E_{\bfV,\mathcal{I}_2}$$
induces a linear map
$$\tau^\ast_{\mathcal{I}_1,\mathcal{I}_2}:\mathcal{F}_{G_{\bfV}}(E_{\bfV,\mathcal{I}_2})\rightarrow \mathcal{F}_{G_{\bfV}}(E_{\bfV,\mathcal{I}_1}).$$

Note that $$\tau_{\mathcal{I}_1,\mathcal{I}_2}^\ast\tau_{\mathcal{I}_2,\mathcal{I}_3}^\ast=\tau_{\mathcal{I}_1,\mathcal{I}_3}^\ast$$ for partitions $\mathcal{I}_1$, $\mathcal{I}_2$ and $\mathcal{I}_3$ of $\mathbb{R}$ such that $\mathcal{I}_2$ is a refinement of $\mathcal{I}_1$ and $\mathcal{I}_3$ is a refinement of $\mathcal{I}_2$.
Hence
$$(\mathcal{F}_{G_{\bfV}}(E_{\bfV,\mathcal{I}}),\tau_{\mathcal{I}_1,\mathcal{I}_2}^\ast)$$ is an inverse system.

\begin{proposition}\label{prop_inverse_5_1}
It holds that $(\bar{\mathcal{F}}_{G_\bfV}(E_{\bfV}),\tau^\ast_{\mathcal{I}})$ is the inverse limit of the inverse system $(\mathcal{F}_{G_{\bfV}}(E_{\bfV,\mathcal{I}}),\tau_{\mathcal{I}_1,\mathcal{I}_2}^\ast).$ \end{proposition}

\begin{proof}
This result is obvious by definitions.

\end{proof}

\subsection{}
Let  $\mathcal{I}_1$ and $\mathcal{I}_2$ be two partitions of $\mathbb{R}$ such that the partition $\mathcal{I}_2$ is a refinement of $\mathcal{I}_1$.

By the following correspondence (\ref{corres_1}) in Section \ref{sec_4.3}
$$
\xymatrix{E_{{\bbV}_{\mathcal{I}_1}}&\hat{E}\ar[l]_-{r_1}\ar[r]^-{r_2}&E_{{\bbV}_{\mathcal{I}_2}}},
$$
there is  a functor 
$$\Psi_{\mathcal{I}_1,\mathcal{I}_2}={r_1}_\flat r^\ast_2 :\cD_{G_{\bbV_{\mathcal{I}_2}}}^b(E_{{\bbV}_{\mathcal{I}_2}})\rightarrow\cD_{G_{\bbV_{\mathcal{I}_1}}}^b(E_{{\bbV}_{\mathcal{I}_1}}),$$
where ${r_1}_\flat$ is the inverse of 
$r^\ast_1$.
Consider all dimension vectors, we get a
homomorphism of $\mathbb{Z}[v,v^{-1}]$-modules
$$\Psi_{\mathcal{I}_1,\mathcal{I}_2}:K_{Q_{\mathcal{I}_2}}\rightarrow K_{Q_{\mathcal{I}_1}}.$$
Note that $\Psi_{\mathcal{I}_1,\mathcal{I}_2}\Phi_{\mathcal{I}_1,\mathcal{I}_2}=id$.

Note that $$\Psi_{\mathcal{I}_1,\mathcal{I}_2}\Psi_{\mathcal{I}_2,\mathcal{I}_3}=\Psi_{\mathcal{I}_1,\mathcal{I}_3}$$ for partitions $\mathcal{I}_1$, $\mathcal{I}_2$ and $\mathcal{I}_3$ of $\mathbb{R}$ such that $\mathcal{I}_2$ is a refinement of $\mathcal{I}_1$ and $\mathcal{I}_3$ is a refinement of $\mathcal{I}_2$.
Hence, 
$$(K_{Q_\mathcal{I}},\Psi_{\mathcal{I}_1,\mathcal{I}_2})$$ is an inverse system.
Let $$(\bar{K}_{A_{\mathbb{R}}},\Psi_{\mathcal{I}})$$ be the inverse limit of this inverse system, where $\bar{K}_{A_{\mathbb{R}}}$ is a
$\mathbb{Z}[v,v^{-1}]$-module and 
$\Psi_{\mathcal{I}}:\bar{K}_{A_{\mathbb{R}}}\rightarrow K_{Q_\mathcal{I}}$ is a homomorphsim.

Similarly, by the following correspondence (\ref{corres_2}) over $\mathbb{F}_q$ in Section \ref{sec_4.3}
$$
\xymatrix{E_{{\bfV}_{\mathcal{I}_1}}&\hat{E}\ar[l]_-{r_1}\ar[r]^-{r_2}&E_{{\bfV}_{\mathcal{I}_2}}},
$$
we have
 a linear map
$$\Psi^F_{\mathcal{I}_1,\mathcal{I}_2}=(r_1^\ast)^{-1}r_2^\ast:\mathcal{F}_{G_{{\bfV}_{\mathcal{I}_2}}}(E_{{\bfV}_{\mathcal{I}_2}})\rightarrow\mathcal{F}_{G_{{\bfV}_{\mathcal{I}_1}}}(E_{{\bfV}_{\mathcal{I}_1}}).$$
Consider all dimension vectors, we get
a linear map
$$\Psi^F_{\mathcal{I}_1,\mathcal{I}_2}:\mathcal{F}_{Q_{\mathcal{I}_2}}\rightarrow \mathcal{F}_{Q_{\mathcal{I}_1}}.$$

\begin{lemma}\label{prop_5.1}
For any partitions $\mathcal{I}_1$ and $\mathcal{I}_2$ of $\mathbb{R}$ such that $\mathcal{I}_2$ is a refinement of $\mathcal{I}_1$,
the follwoing diagram is commutative
$$\xymatrix{
K_{Q_{\mathcal{I}_2}}\ar[r]^-{\Psi_{\mathcal{I}_1,\mathcal{I}_2}}\ar[d]_-{\chi^F_{Q_{\mathcal{I}_2}}}&K_{Q_{\mathcal{I}_1}}\ar[d]^-{\chi^F_{Q_{\mathcal{I}_1}}}\\
\mathcal{F}_{Q_{\mathcal{I}_2}}\ar[r]^-{\Psi^F_{\mathcal{I}_1,\mathcal{I}_2}}&\mathcal{F}_{Q_{\mathcal{I}_1}}.
}$$
\end{lemma}

\begin{proof}
By the properties of sheaf-function correspondence (\cite{Kiehl_Weissauer_Weil_conjectures_perverse_sheaves_and_l'adic_Fourier_transform}), we can get this commutative diagram.

\end{proof}

On the relations between $\Psi^F_{\mathcal{I}_1,\mathcal{I}_2}$ and $\tau^\ast_{\mathcal{I}_1,\mathcal{I}_2}$, we have the following lemma.

\begin{lemma}\label{prop_5.2}
For any partitions $\mathcal{I}_1$ and $\mathcal{I}_2$ of $\mathbb{R}$ such that $\mathcal{I}_2$ is a refinement of $\mathcal{I}_1$, we have the following commutative diagram
$$\xymatrix{
\mathcal{F}_{Q_{\mathcal{I}_2}}\ar[r]^-{\Psi^F_{\mathcal{I}_1,\mathcal{I}_2}}\ar[d]^-{\sigma_{\mathcal{I}_2}^\ast}&\mathcal{F}_{Q_{\mathcal{I}_1}}\ar[d]^-{\sigma_{\mathcal{I}_1}^\ast}
\\
\mathcal{F}_{A_{\mathbb{R}},{\mathcal{I}_2}}\ar[r]^-{\tau_{\mathcal{I}_1,\mathcal{I}_2}^\ast}&\mathcal{F}_{A_{\mathbb{R}},{\mathcal{I}_1}}
.}$$
\end{lemma}

\begin{proof}
Let $V=(\bfV,\phi)$ be a representation of $A_{\mathbb{R}}$ such that $\phi\in E_{\bfV,\mathcal{I}_2}$.

If $\phi\not\in E_{\bfV,\mathcal{I}_1}$, then $$\tau_{\mathcal{I}_1,\mathcal{I}_2}^\ast(1_{\mathcal{O}_V})=0$$
and
$$\Psi^F_{\mathcal{I}_1,\mathcal{I}_2}(\sigma_{\mathcal{I}_2}^\ast)^{-1}(1_{\mathcal{O}_V})=0.$$
Hence, $$\Psi^F_{\mathcal{I}_1,\mathcal{I}_2}(\sigma_{\mathcal{I}_2}^\ast)^{-1}(1_{\mathcal{O}_V})=(\sigma_{\mathcal{I}_1}^\ast)^{-1}\tau_{\mathcal{I}_1,\mathcal{I}_2}^\ast(1_{\mathcal{O}_V}).$$
    
Assume that $\phi\in E_{\bfV,\mathcal{I}_1}$.
Let $V_1=(\bfV_{\mathcal{I}_1},\sigma_{\mathcal{I}_1}(\phi))$ and $V_2=(\bfV_{\mathcal{I}_2},\sigma_{\mathcal{I}_2}(\phi))$.
Note that $\sigma_{\mathcal{I}_2}(\phi)\in\hat{E}$ and let
$V_3=(\bfV_{\mathcal{I}_1},r_1\sigma_{\mathcal{I}_2}(\phi))$.
Then, we have
$$(\sigma_{\mathcal{I}_1}^\ast)^{-1}\tau_{\mathcal{I}_1,\mathcal{I}_2}^\ast(1_{\mathcal{O}_V})=(\sigma_{\mathcal{I}_1}^\ast)^{-1}(1_{\mathcal{O}_V})=1_{\mathcal{O}_{V_1}}$$
and
$$\Psi^F_{\mathcal{I}_1,\mathcal{I}_2}(\sigma_{\mathcal{I}_2}^\ast)^{-1}(1_{\mathcal{O}_V})=\Psi^F_{\mathcal{I}_1,\mathcal{I}_2}(1_{\mathcal{O}_{V_2}})=1_{\mathcal{O}_{V_3}}.$$
Since $V_1\cong V_3$, we have $1_{\mathcal{O}_{V_1}}=1_{\mathcal{O}_{V_3}}$
and 
$$\Psi^F_{\mathcal{I}_1,\mathcal{I}_2}(\sigma_{\mathcal{I}_2}^\ast)^{-1}(1_{\mathcal{O}_V})=(\sigma_{\mathcal{I}_1}^\ast)^{-1}\tau_{\mathcal{I}_1,\mathcal{I}_2}^\ast(1_{\mathcal{O}_V}).$$

Hence, we have
$$\Psi^F_{\mathcal{I}_1,\mathcal{I}_2}(\sigma_{\mathcal{I}_2}^\ast)^{-1}=(\sigma_{\mathcal{I}_1}^\ast)^{-1}\tau_{\mathcal{I}_1,\mathcal{I}_2}^\ast$$
and
$$\sigma_{\mathcal{I}_1}^\ast\Psi^F_{\mathcal{I}_1,\mathcal{I}_2}=\tau_{\mathcal{I}_1,\mathcal{I}_2}^\ast\sigma_{\mathcal{I}_2}^\ast.$$
   
\end{proof}

\begin{proposition}\label{prop_3}
For any partitions $\mathcal{I}_1$ and $\mathcal{I}_2$ of $\mathbb{R}$ such that $\mathcal{I}_2$ is a refinement of $\mathcal{I}_1$, we have the following commutative diagram
$$\xymatrix{K_{Q_{\mathcal{I}_2}}\ar[r]^-{\Psi_{\mathcal{I}_1,\mathcal{I}_2}}\ar[d]_-{\chi_{\mathcal{I}_2}}&K_{Q_{\mathcal{I}_1}}\ar[d]^-{\chi_{\mathcal{I}_1}}\\
\mathcal{F}_{Q_{\mathcal{I}_2}}\ar[r]^-{\tau_{\mathcal{I}_1,\mathcal{I}_2}^\ast}&\mathcal{F}_{Q_{\mathcal{I}_1}}.}$$
\end{proposition}

\begin{proof}
This result is implied by Lemmas \ref{prop_5.1} and \ref{prop_5.2}.

\end{proof}

\begin{theorem}\label{thm_5.5}
There is a  $\mathbb{Z}$-linear map
$$\bar{\chi}_{{A_{\mathbb{R}}}}:\bar{K}_{A_{\mathbb{R}}}\rightarrow\bar{\mathcal{F}}_{A_{\mathbb{R}}}$$such that the following diagram is commutative
\begin{equation}\label{diagram_1}
\xymatrix{\bar{K}_{A_{\mathbb{R}}}\ar[r]^-{\Psi_\mathcal{I}}\ar[d]_-{\bar{\chi}_{{A_{\mathbb{R}}}}}&K_{Q_\mathcal{I}}\ar[d]^-{\chi^F_{Q_\mathcal{I}}}\\
\bar{\mathcal{F}}_{A_{\mathbb{R}}}\ar[r]^-{(\sigma^{\ast}_{\mathcal{I}})^{-1}\tau^\ast_{\mathcal{I}}}&\mathcal{F}_{Q_\mathcal{I}}.}   
\end{equation}
\end{theorem}

\begin{proof}
By Proposition \ref{prop_inverse_5_1}, $(\bar{\mathcal{F}}_{A_{\mathbb{R}}},(\sigma^{\ast}_{\mathcal{I}})^{-1}\tau^\ast_{\mathcal{I}})$ is the inverse limit of $(\mathcal{F}_{Q_\mathcal{I}},\Psi^F_{\mathcal{I}_1,\mathcal{I}_2}).$
At the same time, $(\bar{K}_{A_{\mathbb{R}}},\Psi_{\mathcal{I}})$ is the inverse limit of 
$(K_{Q_\mathcal{I}},\Psi_{\mathcal{I}_1,\mathcal{I}_2})$. By Proposition \ref{prop_3},
there is a $\mathbb{Z}$-linear map$$\bar{\chi}_{{A_{\mathbb{R}}}}:\bar{K}_{A_{\mathbb{R}}}\rightarrow\bar{\mathcal{F}}_{A_{\mathbb{R}}}$$such that the diagram (\ref{diagram_1}) is commutative.

\end{proof}

Since ${\chi^F_{Q_\mathcal{I}}}:K_{Q_\mathcal{I}}\otimes_{\mathbb{Z}[v,v^{-1}]}\mathbb{C}\rightarrow\mathcal{F}_{Q_\mathcal{I}}$ is surjective for all $\mathcal{I}$, the $\mathbb{Z}$-linear map
$\bar{\chi}_{{A_{\mathbb{R}}}}$ induces a surjective $\mathbb{C}$-linear map
$$\bar{\chi}^{\mathbb{C}}_{{A_{\mathbb{R}}}}:\bar{K}_{A_{\mathbb{R}}}\otimes_{\mathbb{Z}[v,v^{-1}]}\mathbb{C}\rightarrow\bar{\mathcal{F}}_{A_{\mathbb{R}}}.$$ 

\subsection{}

For any $A\in K_{A_{\mathbb{R}}}$ and a partition $\mathcal{I}$, there exists $\mathcal{I}_1$ such that $A\in \textrm{Im}(\Phi_{\mathcal{I}_1})$ and
$\mathcal{I}_1$ is a 
refinement of $\mathcal{I}$.
Let $A=\Phi_{\mathcal{I}_1}(A_1)$ and define $$\Theta_{\mathcal{I}}(A)=\Psi_{\mathcal{I},\mathcal{I}_1}(A_1).$$
Let $\mathcal{I}_2$ be another partition such that $A\in \textrm{Im}(\Phi_{\mathcal{I}_2})$ and $\mathcal{I}_2$ is a 
refinement of $\mathcal{I}$.
Assume that $A=\Phi_{\mathcal{I}_1}(A_2)$.
Let $\mathcal{I}_3$ be a partition such that $\mathcal{I}_3$ is 
a refinement of $\mathcal{I}_1$ and $\mathcal{I}_2$. Then, $\Phi_{\mathcal{I}_1,\mathcal{I}_3}(A_1)=\Phi_{\mathcal{I}_2,\mathcal{I}_3}(A_2)$.
Since
$$\Psi_{\mathcal{I},\mathcal{I}_2}(A_2)=\Psi_{\mathcal{I},\mathcal{I}_2}\Psi_{\mathcal{I}_2,\mathcal{I}_3}\Phi_{\mathcal{I}_2,\mathcal{I}_3}(A_2)=\Psi_{\mathcal{I},\mathcal{I}_3}\Phi_{\mathcal{I}_2,\mathcal{I}_3}(A_2)$$
and
$$\Psi_{\mathcal{I},\mathcal{I}_1}(A_1)=\Psi_{\mathcal{I},\mathcal{I}_1}\Psi_{\mathcal{I}_1,\mathcal{I}_3}\Phi_{\mathcal{I}_1,\mathcal{I}_3}(A_1)=\Psi_{\mathcal{I},\mathcal{I}_3}\Phi_{\mathcal{I}_1,\mathcal{I}_3}(A_1),$$
we have $\Psi_{\mathcal{I},\mathcal{I}_1}(A_1)=\Psi_{\mathcal{I},\mathcal{I}_2}(A_2)$.
Hence, $\Theta_{\mathcal{I}}(A)$ is independent of the choice of $\mathcal{I}_1$ and $A_1$, and we get a map
$$\Theta_{\mathcal{I}}: K_{A_{\mathbb{R}}}\rightarrow K_{Q_\mathcal{I}}.$$
By the definition of $\bar{K}_{A_{\mathbb{R}}}$, there is 
a map
$$\Theta_{A_{\mathbb{R}}}: K_{A_{\mathbb{R}}}\rightarrow\bar{K}_{A_{\mathbb{R}}}.$$

\begin{lemma}\label{lemma_1}
For any partitions $\mathcal{I}$ of $\mathbb{R}$, we have the following commutative diagram
$$\xymatrix{K_{A_{\mathbb{R}}}\ar[r]^-{\Theta_{\mathcal{I}}}\ar[d]_-{\chi_{A_{\mathbb{R}}}}&K_{Q_\mathcal{I}}\ar[d]^-{\chi^F_{Q_\mathcal{I}}}\\
\mathcal{F}_{A_{\mathbb{R}}}\ar[r]^-{(\sigma^\ast_{\mathcal{I}})^{-1}\tau^\ast_{\mathcal{I}}}&\mathcal{F}_{Q_\mathcal{I}}.}$$
\end{lemma}

\begin{proof}
For any $A\in K_{A_{\mathbb{R}}}$, there exists $\mathcal{I}_1$ such that $A\in \textrm{Im}(\Phi_{\mathcal{I}_1})$ and
$\mathcal{I}_1$ is a 
refinement of $\mathcal{I}$.
Let $A=\Phi_{\mathcal{I}_1}(A_1)$. By definition, $\Theta_{\mathcal{I}}(A)=\Psi_{\mathcal{I},\mathcal{I}_1}(A_1).$
Then, we have$$\sigma^\ast_{\mathcal{I}}\chi^F_{Q_\mathcal{I}}\Theta_{\mathcal{I}}(A)=\sigma^\ast_{\mathcal{I}}\chi^F_{Q_\mathcal{I}}\Psi_{\mathcal{I},\mathcal{I}_1}(A_1)=\sigma^\ast_{\mathcal{I}}\Psi^F_{\mathcal{I},\mathcal{I}_1}\chi^F_{Q_{\mathcal{I}_1}}(A_1)=\tau^\ast_{\mathcal{I},\mathcal{I}_1}\sigma^\ast_{\mathcal{I}_1}\chi^F_{Q_{\mathcal{I}_1}}(A_1)$$
and
$$\tau^\ast_{\mathcal{I}}\chi_{A_{\mathbb{R}}}(A)=\tau^\ast_{\mathcal{I}}\chi_{A_{\mathbb{R}}}\Phi_{\mathcal{I}_1}(A_1)=\tau^\ast_{\mathcal{I},\mathcal{I}_1}\tau^\ast_{\mathcal{I}_1}(\tau_{\mathcal{I}_1})_!\sigma^\ast_{\mathcal{I}_1}\chi^F_{Q_{\mathcal{I}_1}}(A_1)={\tau}^\ast_{\mathcal{I},\mathcal{I}_1}\sigma^\ast_{\mathcal{I}_1}\chi^F_{Q_{\mathcal{I}_1}}(A_1).$$

Hence, we have $$\sigma^\ast_{\mathcal{I}}\chi^F_{Q_\mathcal{I}}\Theta_{\mathcal{I}}(A)=\tau^\ast_{\mathcal{I}}\chi_{A_{\mathbb{R}}}(A)$$
and get the desired result.

\end{proof}

\begin{proposition}
The following diagram is commutative
$$\xymatrix{K_{A_{\mathbb{R}}}\ar[r]^-{\Theta_{A_{\mathbb{R}}}}\ar[d]_-{\chi_{A_{\mathbb{R}}}}&\bar{K}_{A_{\mathbb{R}}}\ar[d]^-{\bar{\chi}_{A_{\mathbb{R}}}}\\
\mathcal{F}_{A_{\mathbb{R}}}\ar[r]^-{\theta_{A_{\mathbb{R}}}}&\bar{\mathcal{F}}_{A_{\mathbb{R}}}.}$$
\end{proposition}

\begin{proof}
For any $A\in K_{A_{\mathbb{R}}}$, let $A_{\mathcal{I}}=\Theta_{\mathcal{I}}(A)\in K_{Q_{\mathcal{I}}}$. Since $\Psi_{\mathcal{I}_1,\mathcal{I}_2}\Psi_{\mathcal{I}_2}=\Psi_{\mathcal{I}_1}$, we have 
\begin{equation}\label{formula_5.6}
  A_{\mathcal{I}_1}=\Theta_{\mathcal{I}_1}(A)=\Psi_{\mathcal{I}_1}\Theta_{A_\mathbb{R}}(A)=\Psi_{\mathcal{I}_1,\mathcal{I}_2}\Psi_{\mathcal{I}_2}\Theta_{A_\mathbb{R}}(A)=\Psi_{\mathcal{I}_1,\mathcal{I}_2}(A_{\mathcal{I}_2}).  
\end{equation}

Let  $f_{\mathcal{I}}=\sigma^\ast_{\mathcal{I}}\chi^F_{Q_\mathcal{I}}(A_{\mathcal{I}})$. By Proposition \ref{prop_3}, Formula (\ref{formula_5.6}) implies that $f_{\mathcal{I}_1}=\tau^\ast_{\mathcal{I}_1,\mathcal{I}_1}(f_{\mathcal{I}_2})$. By Lemma \ref{prop_5.1}, there exists $f\in\bar{\mathcal{F}}_{A_{\mathbb{R}}}$ such that $\tau^\ast_{\mathcal{I}}(f)=f_{\mathcal{I}}$. Then, we have $\bar{\chi}_{A_{\mathbb{R}}}\Theta_{A_{\mathbb{R}}}(A)=f$  by Theorem \ref{thm_5.5}.

Let $g=\theta_{A_{\mathbb{R}}}\chi_{A_{\mathbb{R}}}(A)$.
Then $$\tau^\ast_{\mathcal{I}}(g)=\tau^\ast_{\mathcal{I}}\chi_{A_{\mathbb{R}}}(A)=\sigma^\ast_{\mathcal{I}}\chi^F_{Q_\mathcal{I}}\Theta_{\mathcal{I}}(A)=f_{\mathcal{I}}=\tau^\ast_{\mathcal{I}}(f)$$
by Lemma \ref{lemma_1}.

Hence, we have $f=g$ and get the desired result.

\end{proof}

Since $\bar{\mathcal{F}}_{A_{\mathbb{R}}}$ is the completion of  $\mathcal{F}_{A_{\mathbb{R}}}$, the set $\bar{K}_{A_{\mathbb{R}}}$ can be viewed as a completion of the geometric version $K_{A_{\mathbb{R}}}$ of Ringel-Hall algebra $\mathcal{F}_{A_{\mathbb{R}}}$.

\subsection{}

Let $V=(\bfV,\phi)$ be a finitely generated representation of $A_{\mathbb{R}}$. Let $\mathcal{I}_1$ and $\mathcal{I}_2$ be partitions of $\mathbb{R}$ such that $\mathcal{I}_2$ is a refinement of $\mathcal{I}_1$ and $\phi\in E_{\bfV,\mathcal{I}_1}$.

Let $V_{\mathcal{I}_1}=(\bfV_{\mathcal{I}_1},\sigma_{\mathcal{I}_1}(\phi))$ and $V_{\mathcal{I}_2}=(\bfV_{\mathcal{I}_2},\sigma_{\mathcal{I}_2}(\phi))$.
Denote by $\mathcal{O}_{V_{\mathcal{I}_1}}$ the orbit in $E_{\bfV_{\mathcal{I}_1}}$  corresponding to $V_{\mathcal{I}_1}$ and $\mathcal{O}_{V_{\mathcal{I}_2}}$ the orbit in $E_{\bfV_{\mathcal{I}_2}}$  corresponding to $V_{\mathcal{I}_2}$.
Denote by $$j_{\mathcal{O}_{V_{\mathcal{I}_1}}}:\mathcal{O}_{V_{\mathcal{I}_1}}\rightarrow E_{\bfV_{\mathcal{I}_1}}, j_{\mathcal{O}_{V_{\mathcal{I}_2}}}:\mathcal{O}_{V_{\mathcal{I}_2}}\rightarrow E_{\bfV_{\mathcal{I}_2}}\textrm{ and }\hat{j}_{\mathcal{O}_{V_{\mathcal{I}_2}}}:\mathcal{O}_{V_{\mathcal{I}_1}}\rightarrow \hat{E}$$ the natural embeddings.

Consider the following commutative diagram
\begin{equation*}
\xymatrix{E_{\bfV_{\mathcal{I}_1}}&\hat{E}\ar[l]_-{r_1}\ar[r]^-{r_2}&E_{\bfV_{\mathcal{I}_2}}\\
\mathcal{O}_{V_{\mathcal{I}_1}}\ar[u]_-{j_{\mathcal{O}_{V_{\mathcal{I}_1}}}}&\mathcal{O}_{V_{\mathcal{I}_2}}\ar[u]_-{\hat{j}_{\mathcal{O}_{V_{\mathcal{I}_2}}}}\ar[l]_-{r_1}\ar[r]^-{id}&\mathcal{O}_{V_{\mathcal{I}_2}}\ar[u]_-{j_{\mathcal{O}_{V_{\mathcal{I}_2}}}}.}
\end{equation*}
Now, we have
\begin{eqnarray*}
\Psi_{\mathcal{I}_1,\mathcal{I}_2}(IC_{\mathcal{O}_{V_{\mathcal{I}_2}}})
&\cong&{r_1}_{\flat}r_2^\ast(j_{\mathcal{O}_{V_{\mathcal{I}_2}}})_{!\ast}(1_{\mathcal{O}_{V_{\mathcal{I}_2}}})\\
&\cong&{r_1}_{\flat}(\hat{j}_{\mathcal{O}_{V_2}})_{!\ast}(1_{\mathcal{O}_{V_{\mathcal{I}_2}}})\\
&\cong&(\hat{j}_{\mathcal{O}_{V_{\mathcal{I}_1}}})_{!\ast}(1_{\mathcal{O}_{V_{\mathcal{I}_1}}})=IC_{\mathcal{O}_{V_{\mathcal{I}_1}}}.   
\end{eqnarray*}
Hence, it holds that
$$\Psi_{\mathcal{I}_1,\mathcal{I}_2}([IC_{\mathcal{O}_{V_{\mathcal{I}_2}}}])=[IC_{\mathcal{O}_{V_{\mathcal{I}_1}}}].$$
By the definition of $\bar{K}_{A_{\mathbb{R}}}$, there exist $[IC_{\mathcal{O}_{V}}]\in\bar{K}_{A_{\mathbb{R}}}$ such that $$\Psi_{\mathcal{I}}([IC_{\mathcal{O}_{V}}])=[IC_{\mathcal{O}_{V_\mathcal{I}}}]$$ for any $\mathcal{I}$ such that $\phi\in E_{\bfV,\mathcal{I}}$.

\begin{theorem}
The set
$$B_{A_{\mathbb{R}}}=\{[IC_{\mathcal{O}_{V}}[n](\frac{n}{2})]|V\in\textrm{rep}_{{K}}(A_\mathbb{R}), n\in\bbZ\}$$ is a $\mathbb{Z}$-basis of $\bar{K}_{A_{\mathbb{R}}}$. 
\end{theorem}

\begin{proof}
This result is direct by the definition of $\bar{K}_{A_{\mathbb{R}}}$ and Theorem \ref{lusztig_canonical}.

\end{proof}

The basis $B_{A_{\mathbb{R}}}$ is called the canonical basis of $\bar{K}_{A_{\mathbb{R}}}$.

\bibliography{mybibfile}

\begin{thebibliography}{10}

\bibitem{Appel_2020}
A.~Appel and F.~Sala.
\newblock Quantization of continuum {K}ac–{M}oody algebras.
\newblock {\em Pure and Applied Mathematics Quarterly}, 16(3):439--493, 2020.

\bibitem{Appel_2022}
A.~Appel, F.~Sala, and O.~Schiffmann.
\newblock Continuum {K}ac–{M}oody algebras.
\newblock {\em Moscow Mathematical Journal}, 22(2):177--224, 2022.

\bibitem{Bernstein_Lunts_Equivariant_sheaves_and_functors}
J.~Bernstein and V.~Lunts.
\newblock {\em Equivariant Sheaves and Functors}.
\newblock Springer, 1994.

\bibitem{2017Interval}
M.~B. Botnan.
\newblock Interval decomposition of infinite zigzag persistence modules.
\newblock {\em Proceedings of the American Mathematical Society}, 145(8):3571--3577, 2017.

\bibitem{2018Decomposition}
M.~B. Botnan and W.~Crawley-Boevey.
\newblock Decomposition of persistence modules.
\newblock {\em Proceedings of the American Mathematical Society}, 148(11):4581--4596, 2020.

\bibitem{Crawley-Boevey90}
W.~Crawley-Boevey.
\newblock Lectures on representations of quivers, 1992.
\newblock https://www.math.uni-bielefeld.de/\textasciitilde wcrawley/quivlecs.pdf.

\bibitem{Crawley2015Decomposition}
W.~Crawley-Boevey.
\newblock Decomposition of pointwise finite-dimensional persistence modules.
\newblock {\em Journal of Algebra and its Applications}, 14(5):1550066, 2015.

\bibitem{Green1995}
J.~A. Green.
\newblock {H}all algebras, hereditary algebras and quantum groups.
\newblock {\em Inventiones Mathematicae}, 120(1):361--377, 1995.

\bibitem{2020Decomposition_Hanson_Rock}
E.~J. Hanson and J.~D. Rock.
\newblock Decomposition of pointwise finite-dimensional ${S}^1$ persistence modules.
\newblock {\em Journal of Algebra and its Applications}, 23(03):2450054, 2024.

\bibitem{IGUSA_ROCK_TODOROV_2022}
K.~Igusa, J.~D. Rock, and G.~Todorov.
\newblock Continuous quivers of type {$A$} ({III}) embeddings of cluster theories.
\newblock {\em Nagoya Mathematical Journal}, 247:653--689, 2022.

\bibitem{Igusa2022Continuous}
K.~Igusa, J.~D. Rock, and G.~Todorov.
\newblock Continuous quivers of type ${A}$ ({I}) foundations.
\newblock {\em Rendiconti del Circolo Matematico di Palermo Series 2}, 72(2):833--868, 2023.

\bibitem{Kiehl_Weissauer_Weil_conjectures_perverse_sheaves_and_l'adic_Fourier_transform}
R.~Kiehl and R.~Weissauer.
\newblock {\em Weil conjectures, perverse sheaves and l'adic {F}ourier transform}.
\newblock Springer, 2001.

\bibitem{Lusztig_Canonical_bases_arising_from_quantized_enveloping_algebra}
G.~Lusztig.
\newblock Canonical bases arising from quantized enveloping algebras.
\newblock {\em Journal of the American Mathematical Society}, 3(2):447--498, 1990.

\bibitem{Lusztig_Quivers_perverse_sheaves_and_the_quantized_enveloping_algebras}
G.~Lusztig.
\newblock Quivers, perverse sheaves, and quantized enveloping algebras.
\newblock {\em Journal of the American Mathematical Society}, 4(2):365--421, 1991.

\bibitem{Lusztig_Canonical_bases_and_Hall_algebras}
G.~Lusztig.
\newblock Canonical bases and {H}all algebras.
\newblock In {\em Representation Theories and Algebraic Geometry}, pages 365--399. Springer, 1998.

\bibitem{Lusztig_Introduction_to_quantum_groups}
G.~Lusztig.
\newblock {\em Introduction to quantum groups}.
\newblock Springer, 2010.

\bibitem{Persistence}
S.~Y. Oudot.
\newblock {\em Persistence theory: from quiver representations to data analysis}.
\newblock American Mathematical Society, 2015.

\bibitem{Ringel_Hall_algebras_and_quantum_groups}
C.~M. Ringel.
\newblock Hall algebras and quantum groups.
\newblock {\em Inventiones Mathematicae}, 101(1):583--591, 1990.

\bibitem{rock2023A_2}
J.~D. Rock.
\newblock Continuous quivers of type {$A$} ({II}) the {A}uslander-{R}eiten space.
\newblock {\em arXiv:1910.04140}, 2019.

\bibitem{rock2023A_3}
J.~D. Rock.
\newblock Continuous quivers of type {$A$} ({IV}) continuous mutation and geometric models of {E}-clusters.
\newblock {\em Algebras and Representation Theory}, 26(5):2255--2288, 2023.

\bibitem{rock2023continuous_Nakayama}
J.~D. Rock and S.~Zhu.
\newblock Continuous {N}akayama representations.
\newblock {\em Applied Categorical Structures}, 31(5):44, 2023.

\bibitem{Sala_2019}
F.~Sala and O.~Schiffmann.
\newblock The circle quantum group and the infinite root stack of a curve.
\newblock {\em Selecta Mathematica}, 25(5):77, 2019.

\bibitem{Sala_2021}
F.~Sala and O.~Schiffmann.
\newblock Fock space representation of the circle quantum group.
\newblock {\em International Mathematics Research Notices}, 2021(22):17025--17070, 2021.

\bibitem{Schiffmann_Lectures1}
O.~Schiffmann.
\newblock Lectures on {H}all algebras.
\newblock {\em arXiv:0611617}, 2006.

\bibitem{Schiffmann_Lectures2}
O.~Schiffmann.
\newblock Lectures on canonical and crystal bases of {H}all algebras.
\newblock {\em arXiv:0910.4460}, 2009.

\bibitem{XIAO1997100}
J.~Xiao.
\newblock {D}rinfeld double and {R}ingel-{G}reen theory of {H}all algebras.
\newblock {\em Journal of Algebra}, 190(1):100--144, 1997.

\bibitem{XXZ2022254}
J.~Xiao, F.~Xu, and M.~Zhao.
\newblock {R}ingel-{H}all algebras beyond their quantum groups {I}: Restriction functor and {G}reen formula.
\newblock {\em Algebras and Representation Theory}, 22(5):1299--1329, 2019.

\end{thebibliography}

\end{document}